\documentclass{amsart}

\RequirePackage{color}

\def\smallskip{\vskip\smallskipamount}
\def\medskip{\vskip\medskipamount}
\def\bigskip{\vskip\bigskipamount}

\usepackage{amsmath,amsthm,bm,mathrsfs,amscd,tikz}
\usepackage{ytableau}
\usepackage{comment}
\usepackage{mathdots}
\usepackage[all]{xy}
\usepackage{graphicx}
\usepackage[colorinlistoftodos]{todonotes}
\usepackage{enumerate}
\usepackage{feynmp-auto}
\usepackage[english]{babel}
\usepackage[utf8]{inputenc}
\usepackage{float}
\usepackage{tikz}
\usepackage{tikz-cd}
\usepackage{amssymb}
\usepackage{mathtools}
\usepackage{soul}
\usetikzlibrary{decorations.pathmorphing}
\usepackage{url}
\usepackage[colorlinks=false, linktocpage=true]{hyperref}

\usepackage[utf8]{inputenc}

\newcommand{\nocontentsline}[3]{}
\let\origcontentsline\addcontentsline
\newcommand\stoptoc{\let\addcontentsline\nocontentsline}
\newcommand\resumetoc{\let\addcontentsline\origcontentsline}

\newtheoremstyle{thmstyle}{}{}{\itshape}{}{\bfseries}{ }{5pt}{}
\newtheoremstyle{exstyle}{}{}{}{}{\bfseries}{ }{5pt}{}
\newtheoremstyle{defstyle}{}{}{}{}{\bfseries}{ }{5pt}{}
\newtheoremstyle{remstyle}{}{}{}{}{\bfseries}{ }{5pt}{}

\theoremstyle{thmstyle}
\newtheorem{thm}{Theorem}[section]
\newtheorem{theorem}[thm]{Theorem}

\newtheorem{lemma}[thm]{Lemma}

\newtheorem{proposition}[thm]{Proposition}

\newtheorem{corollary}[thm]{Corollary}

\theoremstyle{exstyle}

\theoremstyle{defstyle}

\newtheorem{definition}[thm]{Definition}

\newtheorem{def-prop}[thm]{Definition-Proposition}
\newtheorem{def-lem}[thm]{Definition-Lemma}
\newtheorem{rem-convention}[thm]{Remark-Convention}
\newtheorem{def-note}[thm]{Definition-Notation}

\theoremstyle{remstyle}

\newtheorem{remark}[thm]{Remark}

\theoremstyle{remstyle}

\newcommand{\Hom}{\operatorname{Hom}}
\newcommand{\Fac}{\operatorname{Fac}}

\DeclareMathOperator*{\rad}{rad}

\DeclareMathOperator*{\modu}{mod}

\DeclareMathOperator*{\Sub}{Sub}

\DeclareMathOperator*{\Filt}{Filt}

\DeclareMathOperator*{\ind}{ind}

\DeclareMathOperator*{\brick}{brick}

\DeclareMathOperator*{\tors}{tors}

\DeclareMathOperator*{\torf}{torf}

\DeclareMathOperator*{\JI}{JI}
\DeclareMathOperator*{\MI}{MI}
\DeclareMathOperator*{\Hasse}{Hasse}
\DeclareMathOperator*{\LM}{LM}

\newcommand{\cupdot}{\mathbin{\mathaccent\cdot\cup}}

\makeatletter
\newcommand*{\doublerightarrow}[2]{\mathrel{
  \settowidth{\@tempdima}{$\scriptstyle#1$}
  \settowidth{\@tempdimb}{$\scriptstyle#2$}
  \ifdim\@tempdimb>\@tempdima \@tempdima=\@tempdimb\fi
  \mathop{\vcenter{
    \offinterlineskip\ialign{\hbox to\dimexpr\@tempdima+1em{##}\cr
    \rightarrowfill\cr\noalign{\kern.5ex}
    \rightarrowfill\cr}}}\limits^{\!#1}_{\!#2}}}
\newcommand*{\triplerightarrow}[1]{\mathrel{
  \settowidth{\@tempdima}{$\scriptstyle#1$}
  \mathop{\vcenter{
    \offinterlineskip\ialign{\hbox to\dimexpr\@tempdima+1em{##}\cr
    \rightarrowfill\cr\noalign{\kern.5ex}
    \rightarrowfill\cr\noalign{\kern.5ex}
    \rightarrowfill\cr}}}\limits^{\!#1}}}
\makeatother

\makeatletter
\newcommand{\doublewidetilde}[1]{{%
  \mathpalette\double@widetilde{#1}%
}}
\newcommand{\double@widetilde}[2]{%
  \sbox\z@{$\m@th#1\widetilde{#2}$}%
  \ht\z@=.9\ht\z@
  \widetilde{\box\z@}%
}
\makeatother

\begin{document}

\title[Left modularity and extremality for infinite lattices]{Left modularity and extremality for (some) infinite lattices}

\author[Sota Asai, Osamu Iyama, Kaveh Mousavand, Charles Paquette]{Sota Asai, Osamu Iyama, Kaveh Mousavand, Charles Paquette} 

\address{Sota Asai: Graduate School of Mathematical Sciences, the University of Tokyo, Japan}
\email{sotaasai@g.ecc.u-tokyo.ac.jp}
\address{Osamu Iyama: Graduate School of Mathematical Sciences, the University of Tokyo, Japan}
\email{iyama@ms.u-tokyo.ac.jp}
\address{Kaveh Mousavand: Representation Theory and Algebraic Combinatorics Unit, Okinawa Institute of Science and Technology (OIST), Japan}
\email{mousavand.kaveh@gmail.com}
\address{Charles Paquette: Department of Mathematics and Computer Science, Royal Military College of Canada, Kingston ON, Canada}
\email{charles.paquette.math@gmail.com}



\begin{abstract} 
For some important families of complete infinite lattices, we study some generalizations of two fundamental notions which are mostly treated for finite lattices. 
Specifically, for well-separated $\kappa$-lattices, and also for weakly atomic completely semidistributive lattices, we generalize the notions of left modularity and extremality. 
These two families of lattices coincide if restricted to finite lattices, but are distinct when infinite lattices are also included. 
For both families, we prove that extremality and left modularity imply each other. 
Furthermore, for weakly atomic completely semidistributive lattices, we give several conceptual characterizations of left modular elements, and show that the set of left modular elements form a complete distributive sublattice. 
Our results, combined with some recent work on finite lattices, imply that the weakly atomic completely semidistributive lattices that are left modular (or extremal) generalize the semidistributive trim lattices; from finite to infinite lattices. 
We then apply our results to the lattice of torsion classes of finite dimensional algebras, which are known to fall in the intersection of the two families treated in our work. For an algebra $A$, we obtain that the lattice of torsion classes is left modular (equivalently, extremal) if and only if $A$ is brick-directed. This leads to an abundance of concrete examples and non-examples.
\end{abstract}

\maketitle

\tableofcontents

\newpage

\section{Summary and main results}\label{Sec: Summary of Main Results}
\stoptoc

\subsection{Summary}
Here we briefly present our motivations and main contributions. For notation, terminology, and background material, see Section \ref{Section: Preliminaries and Background} and references therein. 
Throughout, unless otherwise specified, $(L,\leq)$ denotes a complete lattice, which is not necessarily finite. When there is no confusion, we simply denote a (complete) lattice by $L$. The treatment of arbitrary lattices is known to be formidable, and most studies assume some additional conditions on the lattice. 
Some notable recent developments in lattice theory have been inspired by the remarkable properties of lattices of torsion classes of finite dimensional algebras, leading to important results beyond finite lattices; see \cite{DI+, RST} and the references therein. 

\medskip

Distributive lattices are known to satisfy many nice properties, but they form a rather small family. For finite lattices, several generalizations of the distributivity property have been studied; see \cite{Mu} and the references therein. Notably, extremal lattices and left modular lattices are two of the most studied generalizations of finite distributive lattices; respectively introduced in \cite{Ma} and \cite{LS}. Neither of these properties implies the other one. 
However, for a semidistributive finite lattice $L$, recent results of \cite{TW, Mu}, together, show that $L$ is left modular if and only if $L$ is extremal. That being the case, $L$ is a trim lattice \cite{Th}. 
This establishes elegant connections between some generalizations of distributive lattices; albeit under the assumptions that $L$ is finite and semidistributive. 
Meanwhile, as proved in \cite{Ma}, every finite lattice can be embedded as a subinterval into an extremal lattice, hence a complete classification of extremal lattices seems out of reach.

\medskip

Thanks to the fruitful links between representation theory and lattice theory, a rich collection of lattices arises from the former area, more specifically in connection with the study of torsion classes, bricks, and related topics. The treatment of lattice of torsion classes has led to many fundamental results in a broader context, including ``The fundamental theorem of finite semidistributive lattices"; see \cite{RST}. 
For any finite dimensional associative $k$-algebra $A$, the lattice of torsion classes in $\modu A$, denoted by $(\tors A, \subseteq)$, is known to be a completely semidistributive weakly atomic well-separated $\kappa$-lattice \cite{DI+}. 
For a lattice $(L,\leq)$, these properties can be seen as a kind of discreteness of $L$; however, $\tors A$ is not necessarily finite, nor even countable. 
In fact, $\tors A$ is finite if and only if $A$ is brick-finite (see \cite{AIR, DIJ}). 

\medskip

Building upon our recent work \cite{AI+} on the lattice theory of torsion classes, we go beyond the representation-theoretical context and study some families of infinite lattices that possess a subset of the properties of $(\tors A, \subseteq)$. 
More specifically, we are primarily interested in infinite lattices and treat those which are well-separated $\kappa$-lattices, and also those lattices that are completely semidistributive weakly atomic. 
For both classes, we generalize the classical notions of extremality and left modularity. 
Moreover, in our setting, we study the implications between our generalized notions of left modularity and extremality. 
Before presenting our main results and putting our work and some more recent progress into a broader perspective, we give an outline of this manuscript.

\subsection{Outline} The next subsection is comprised of our main results (Subsection \ref{Subsection: Main Results}). 
In Section \ref{Section: Preliminaries and Background}, after setting up our notation and terminology, we collect some tools from representation theory and lattice theory, used throughout the paper. 
In Section \ref{Section: Left modularity and extremality for well-separated kappa-lattices}, we primarily focus on well-separated $\kappa$-lattices. 
More specifically, in Subsection \ref{Subsection: Generalization of left modularity and extremality for kappa-lattices} we introduce the generalized notions of left modularity and extremality for arbitrary (complete) lattices. Restricted to finite lattices, we verify that our generalizations coincide with the notions originally defined in \cite{LS} and \cite{Mu} for finite lattices. 
Then, in Subsection \ref{Subsection: Equivalence of left modularity and extremality for kappa-lattices}, we show that our generalized notions of left modularity and extremality are equivalent over well-separated $\kappa$-lattices; this is the counterpart of some recent results on finite semidistributive lattices in \cite{Mu, TW}. 
In Section \ref{Section: Left modularity and Extremality for weakly atomic lattices}, we shift our attention to weakly atomic and completely semidistributive lattices. 
In Subsection \ref{Subsection: Left modularity in weakly atomic completely semidistributive lattices}, we give several alternative characterizations of left modular elements in such lattices, and in Subsection \ref{Subsection: Extremality in weakly atomic completely semidistributive lattices} we compare the new notions of left modularity and extremality; we prove that they are in fact equivalent. 
Then, in Subsection \ref{Subsection: The labelling quiver}, we present a significantly less technical realization of left modularity (equivalently, extremality) in terms of labelling quivers. 
In Section \ref{Sec: Brick-directed algebras, Left modularity, and Extremality}, via lattices of torsion classes, we link our new results to the study of brick-directed algebras. This ultimately generalizes some earlier work in this direction and generates a large collection of explicit examples for extremal and left modular infinite lattices.
Our results on the infinite lattices treated in this work also inspire some new problems, including those formulated in Remarks \ref{Remark-Question} and \ref{Remark-Question: on the infinite cardinal}. 

\subsection{Main Results}\label{Subsection: Main Results}
Throughout, $(L,\leq )$ always denotes a complete lattice. An element $x$ of $L$ is called \emph{left modular} if for any $y\leq z$, we have $(y\vee x)\wedge z=y\vee(x\wedge z)$. By $\LM(L)$, denote the set of all left modular elements in $L$. Following \cite{LS}, a finite lattice is called \emph{left modular} if it has an inclusion-maximal chain of left modular elements. 
Moreover, a finite lattice is said to be \emph{extremal} if it has exactly $n$ join-irreducibles as well as $n$ meet-irreducibles, and a chain of length $n$ (namely, the chain has $n+1$ elements); see \cite{Ma}. 
All terms that are not explicitly defined in this section are recalled or introduced in Section \ref{Section: Preliminaries and Background}.

\medskip

We propose an extension of left modularity and extremality beyond finite lattices. 
Our proposed generalization of left modularity is the obvious candidate: We say a (possibly-infinite but complete) lattice $L$ is \emph{\textbf{left modular}} if it has an inclusion-maximal chain of left modular elements. 
In contrast, our generalization of extremality is a bit more technical (for some clarification on the technicality and simplification of this notion in some cases, see Remark \ref{Rem: Technicality in extremality} and Section \ref{Section: Left modularity and Extremality for weakly atomic lattices}): We say an arbitrary lattice $L$ is \emph{\textbf{extremal}} if there exist a bijection $\lambda: {\JI^c}(L)\to {\MI^c}(L)$ and a maximal chain $\mathcal{X}$ in $L$ such that, for each $x\in \mathcal{X}$, if $j \in {\JI^c}(L)$ satisfies $j\not\leq x$, then $x\leq \lambda(j)$, and further, if $x\lessdot y$ for two elements $x,y$ of $\mathcal{X}$, then there is a unique completely join-irreducible $j$ which is below $y$ (i.e., $j<y$) but not below $x$, and a unique completely meet-irreducible $m$ which is above $x$ (i.e., $x<m$) but not above $y$, and $\lambda(j)=m$. 
If $L$ is finite, it is evidently left modular in our sense if and only if it is left modular in the classical sense. Also, $L$ is extremal in our sense if and only if it is extremal in the usual sense; Lemma \ref{Lem: generalized extremal for finite lattices}. 

\medskip

In Section \ref{Section: Left modularity and extremality for well-separated kappa-lattices}, for the well-separated $\kappa$-lattices, we prove that our more general notions of extremality and left modularity are equivalent. Then, in Section \ref{Section: Left modularity and Extremality for weakly atomic lattices}, we show a similar result for weakly atomic completely semidistributive lattices, where the notion of extremality is realized via a noticeably less technical condition. 
Before we state our first result, let us emphasize that, the two families of lattices treated in Sections \ref{Section: Left modularity and extremality for well-separated kappa-lattices} and \ref{Section: Left modularity and Extremality for weakly atomic lattices} are different; see Remark \ref{Rem: comparison between our two implications}. 
More specifically, neither of the assertions of the following theorem implies the other one. 

\begin{theorem}[Theorems \ref{Thm: for well-sep kappa lattice, left modular equi to extremal} and \ref{Thm: Equiv left modular & Extremal for weakly atomic completely semidistributive lattices}] \label{Thm: Equivalences for infinite lattices}
Let $(L,\leq)$ be a complete lattice. 
\begin{enumerate}
    \item Let $L$ be a well-separated $\kappa$-lattice. Then, $L$ is left modular if and only if $L$ is extremal.
    \item Let $L$ be a weakly atomic completely semidistributive lattice. Then, $L$ is left modular if and only if $L$ is extremal.
\end{enumerate}
\end{theorem}

To compare with similar results on finite lattices, note that, for a finite lattice $L$, it is known that $L$ is a well-separated $\kappa$-lattice if and only if $L$ is (weakly atomic completely) semidistributive; see \cite[Theorem 3.1]{RST}. 
Hence, the families of lattices considered in parts (1) and (2) of Theorem \ref{Thm: Equivalences for infinite lattices} coincide for finite lattices, but they are known to differ in the infinite case; see Remark \ref{Rem: comparison between our two implications}. 
Moreover, if a finite lattice $L$ belongs to either (thus, both) of these families, then $L$ is left modular if and only if $L$ is extremal; see Theorem \ref{Thm: trim=extremal=left modular} for these facts on finite lattices from \cite[Theorem 1.4]{TW} and \cite[Corollary 3.4]{Mu}.
Hence, Theorem \ref{Thm: Equivalences for infinite lattices} can be seen as generalizations of these implications, where we also include the infinite lattices specified above. 

\medskip

For a weakly atomic completely semidistributive lattice $L$, let $\kappa$ be the bijection between completely join-irreducible and completely meet-irreducible elements in $L$, given by $\kappa: \JI^c(L) \to \MI^c(L)$, where $\kappa(j):=\max \{y\mid j\wedge y=j_*\}$, for $j \in \JI^c(L)$. 
Also, recall that for such a lattice $L$, the \emph{labelling quiver} of $L$, denoted by $Q_L$, is the directed graph whose vertices are elements of $\JI^c(L)$, and for $i, j \in L$, there is an arrow $i \to j$ if $i \ne j$ and $i \not \le \kappa(j)$. For details, see Section \ref{Subsection: The labelling quiver}.
A set $S$ of vertices of $Q_L$ is \emph{successor-closed} if for every arrow $i\to j$ in $Q_L$, if $i \in S$ then $j \in S$. By ${\rm succ}(Q_L)$ we denote the set consisting of all successor-closed sets of vertices of $Q_L$.

\begin{theorem}[Theorem \ref{Thm: bijection between LM(L) and successor-closed sets}]\label{Thm: successor-closed sets and set of left modular elements}\label{Thm: LM(L) and successor-closed sets in Introduction}
Let $L$ be a weakly atomic completely semidistributive lattice. 
Then,
    $\varphi: 
    \LM(L) \to {\rm succ}(Q_L)$, where $\varphi(t) := \{j \in \JI^c(L) \mid j \le t\}$,
    is a bijection, whose inverse is given by $\psi(S) = \bigvee S$.
\end{theorem}

Our next result gives various characterizations of the left modular elements in weakly atomic completely semidistributive lattices. Such lattices are known to be $\kappa$-lattice (\cite[Theorem 3.1]{RST}), but not necessarily well-separated; see Remark \ref{Rem: comparison between our two implications}. The following theorem, combined with Theorem \ref{Thm: successor-closed sets and set of left modular elements}, also implies that ${\rm succ}(Q_L)$ inherits the structure of a completely distributive lattice from $\LM(L)$.

\begin{theorem}[Propositions \ref{Prop: Characterization of left modular}, \ref{Prop: left modular and cover-relation equivalences}, and \ref{Prop: LM(L) is distributive sublattice}]
\label{Thm: left modular elements in infinite lattices}
Let $(L,\leq)$ be a weakly atomic completely semidistributive lattice. For $x\in L$, the following are equivalent:
\begin{enumerate}
    \item $x$ is left modular.
    \item For each $j\in{\JI^c}(L)$, either $j\le x$ or $x\le \kappa(j)$ holds.
    \item For any cover relation $y\lessdot z$ in $L$, we have $x \vee y = x \vee z$ or $x \wedge y = x \wedge z$ but not both.
    \item $x = \bigvee_{j \in S}j$ for a successor-closed subquiver $S$ of the labelling quiver of $L$.
\end{enumerate}
Moreover, $\LM(L)$ forms a complete sublattice of $L$ which is completely distributive.
\end{theorem}

We further observe that, using the labelling defined in \cite[Lem. 3.11 \& Rem. 3.12]{RST}, the characterization of left modular elements in part $(2)$ of Theorem \ref{Thm: left modular elements in infinite lattices} is analogous to that of \cite[Theorem 1.4 (iii)]{LS} given in terms of cover relations. For more details, see Remark \ref{Rem: comparison with LS}. 

\medskip

Recall that the \emph{spine} of a finite lattice $L$ is the set of elements that lie on some maximum length chain in $L$. 
If $L$ is trim, the spine of $L$ is the same as $\LM(L)$ and forms a distributive sublattice of $L$; see \cite[Theorem 3.7]{TW}. 
In our setting of  weakly atomic completely semidistributive lattices, which includes infinite lattices, thanks to Theorems \ref{Thm: Equivalences for infinite lattices} and \ref{Thm: left modular elements in infinite lattices}, if $L$ is extremal (equivalently, left modular) we have that $\LM(L)$ is always a distributive sublattice of $L$; this can be seen as the generalization of the spine, originally defined for finite lattices.

\medskip

Observe that it follows from the above theorem that if $L$ is extremal, then $Q_L$ has no oriented cycles. In this case, we can give $Q_L$ the structure of a poset as follows: for $x, y \in \JI^c(L)$, we write $x \preceq y$ if there is an oriented path from $y$ to $x$. With this structure, extremal chains are in bijective correspondence with linear extensions of $\preceq$ as follows. For more details, see Subsection \ref{Subsection: The labelling quiver}.

\begin{theorem}[Theorem \ref{Thm: labelling quiver, linear extensions and extremal chains}]\label{Thm: labelling quiver, linear extensions and extremal chains in Introduction}
Let $(L,\leq)$ be a weakly atomic completely semidistributive lattice. If $L$ is extremal, there exists a bijection between linear extensions of $\preceq$ and extremal chains in $L$. Given a linear extension $\mathcal{L}$ of $(Q_L, \preceq)$, the set $D(\mathcal{L})$ of downsets of $\mathcal{L}$ yields the extremal chain $\{\bigvee D \mid D \in D(\mathcal{L})\}$.
\end{theorem}

\medskip

Adapting this lattice theoretical framework in the representation theory of finite dimensional algebras, our recent study of brick-directed algebras \cite{AI+} allows us to generalize some of our results on the lattice theory of torsion classes. In fact, we give a full characterization of those algebras whose lattice of torsion classes are left modular; equivalently, extremal. 
Before we state the next result, we reiterate that, for any algebra $A$, the lattice of torsion classes satisfies all the assumptions of those lattices treated in Theorems \ref{Thm: Equivalences for infinite lattices} and \ref{Thm: left modular elements in infinite lattices}; see \cite{DI+, RST}. 
The next theorem gives a characterization of left modularity (equivalently, extremality) of $\tors A$ in terms of the brick-splitting torsion pairs introduced in \cite{AI+}. 
Recall that for $\mathcal{T} \in \tors A$, we say $(\mathcal{T}, \mathcal{T}^{\perp})$ is a \emph{brick-splitting} torsion pair if every brick $B$ in $\modu A$ belongs either to $\mathcal{T}$ or $\mathcal{T}^{\perp}$. That being the case, $\mathcal{T}$ is called a \emph{brick-splitting} torsion class. For the definition of the brick quiver of an algebra, and the successor-closed subquiver, see, respectively, Section \ref{Subsection:Brick-directed Algebras} and Section \ref{Subsection: The labelling quiver}.

\begin{theorem}[Propositions \ref{Prop: brick-splitting, left modular, successor-closed} \& \ref{Prop: tors A being left modular, extremal, and a chain}]\label{thm: characterization by brick-splitting}
Let $A$ be a connected finite dimensional algebra. 
For any torsion class $\mathcal{T} \in \tors A$, the following are equivalent:
\begin{enumerate}
    \item $\mathcal{T}$ is a left modular element of the lattice $\tors A$.
    \item {The bricks in $\mathcal{T}$ forms a successor-closed subquiver of the brick quiver of $A$.}
\end{enumerate}
Therefore, $\tors A$ is left modular (equivalently, extremal) if and only if there exists a maximal chain $\{\mathcal{T}_i\}_{i\in I}$ in $\tors A$ such that each $\mathcal{T}_i$ is a brick-splitting torsion class.
\end{theorem}

From Theorems \ref{Thm: left modular elements in infinite lattices} and \ref{thm: characterization by brick-splitting}, we obtain a less technical and completely algebraic characterization of the left modularity (equivalently, extremality) of $\tors A$. Before stating this result, recall that a cycle $X_0\xrightarrow{f_1} X_1\xrightarrow{f_2}  \cdots \xrightarrow{f_{m-1}}  X_{m-1}\xrightarrow{f_m}  X_m=X_0$ in $\modu A$ is called a \emph{brick-cycle} if $X_i \in \modu A$ is a brick, for all $0\leq i\leq m$, and the morphisms $f_i$ are nonzero and non-invertible (for details, see Section \ref{Subsection:Brick-directed Algebras}). An algebra $A$ is said to be \emph{brick-directed} if there exists no brick-cycle in $\modu A$. 

\begin{corollary}[Corollary \ref{Cor: brick-directed algs and left modularity=extremality}]
\label{Cor: left modularity and brick-directed algs}
For an algebra $A$, the following are equivalent:
\begin{enumerate}
    \item $A$ is brick-directed.
    \item $\tors A$ is a left modular lattice.
    \item $\tors A$ is an extremal lattice.
    \item The brick quiver of $A$ has no oriented cycles.
\end{enumerate}
\end{corollary}

As the first consequence of Corollary \ref{Cor: left modularity and brick-directed algs}, combined with our result \cite[Prop. 4.9]{AI+}, we obtain a full classification of those path algebras $kQ$ for which $\tors kQ$ is left modular (equivalently, extremal) in the more general sense. We further use our representation-theoretical results on self-injective algebras \cite[Cor. 3.16]{AI+}, together with an important dictionary developed in \cite{IT} and \cite{Mi}, to conclude a result on the left modular elements of Weyl groups of simply laced Dynkin diagrams. 
The first part of the next result was shown in \cite[Prop. 4.9]{AI+}. Here, we formualte it in the lattice-theoretical framework adopted and developed in our current work.

\begin{corollary}[{Corollary \ref{Cor: On tors kQ and Weyl group}}]\label{Cor: On tors kQ and Weyl group in Introduction}
Let $Q$ be a connected acyclic quiver. Then
\begin{enumerate}
    \item The lattice of torsion classes $\tors kQ$ is left modular if and only if $Q$ is a Dynkin quiver or $Q$ is the Kronecker quiver.
    \item For a Dynkin quiver $Q$ and the associated Weyl group lattice  $(W_Q,\leq)$ with respect to the weak order, $\hat{0}$ and $\hat{1}$ are the only left modular elements.
\end{enumerate}

\end{corollary}

As another implication of Corollary \ref{Cor: left modularity and brick-directed algs}, together with an explicit construction of brick-directed algebras given in \cite[Section 4.3]{AI+}, we have the following corollary. This particularly results in an abundance of concrete examples of left modular and extremal lattices of arbitrary length. 

\begin{corollary}[Corollary \ref{Cor: left modular and extremal lattices of big cardinality}] \label{Cor: left modular and extremal lattices of big cardinality in Introduction}
With the same notation as above,
\begin{enumerate}
    \item For every $m>4$, there exist infinitely many connected algebras $A$ for which $\tors A$ is a left modular (equivalently, extremal) lattice of length $m$.

    \item For every cardinality $\aleph$, there exist a left modular and extremal lattice with $2^{\aleph}$ many maximal chain of size ${\aleph}$ consisting of left modular elements.

    \item For every $\aleph$, there exists an extremal lattice for which $\LM(L)$ is of size $2^{\aleph}$.
\end{enumerate}
\end{corollary}

\bigskip

\resumetoc

\section{Preliminaries and background} \label{Section: Preliminaries and Background}

In this section, we briefly recall some facts and terminology from lattice theory and representation theory used in our work. For more details and rudimentary materials in lattice theory, see \cite{Gr}, \cite{DP}, and references therein. Moreover, \cite{ASS} and references therein contain the required materials from representation theory

\subsection{Lattice theory}\label{Subsec: Lattice Theory}
Let $(P,\leq)$ be a non-empty (possibly infinite) partially ordered set. Recall that $x<z$ in $P$ is said to be a \emph{cover} relation, denoted by $x\lessdot z$, provided that for each $y\in P$, if $x\leq y \leq z$, we have either $x=y$ or $y=z$. Moreover, $\Hasse(P)$ denotes the \emph{Hasse quiver} of $P$, which is a directed graph whose vertex set is $P$, and it has an arrow $x\rightarrow z$ for each cover relation $x\lessdot z$. 

\medskip

Recall that $(P,\leq)$ is a \emph{complete lattice} if for any subset $S$ in $P$, there exists a unique element of $P$, smallest with the property of being larger than or equal to all elements of $S$, called the \emph{join} of $S$ and denoted by $\bigvee S$; as well as a unique element of $P$, largest with the property of being smaller than or equal to all elements of $S$, called the \emph{meet} of $S$ and denoted by $\bigwedge S$. 
To denote a lattice, we often use $(L,\leq)$, and for $x$ and $y$ in $L$, the join and meet are respectively denoted by $x\vee y$ and $x\wedge y$. 
In the following, unless specified otherwise, $(L,\leq)$ always denotes a complete lattice, and $\hat{1}$ and $\hat{0}$ respectively denote the greatest element and the least element of $L$ (sometimes called the maximum and minimum of $L$); that is, $\hat{1}:= \bigwedge \emptyset =\bigvee L$ and  $\hat{0}:= \bigwedge L=\bigvee \emptyset$. 
When there is no confusion, we often suppress $\leq$ in the notation and simply use $L$ instead of $(L,\leq)$. 
We also recall that $L$ is said to be \emph{weakly atomic} if, for all $x < y$, there exist $u$ and $v$ with $x \leq  u \lessdot v \leq y$ (equivalently, $\Hasse[x,y]$ contains at least one arrow). Every finite lattice is weakly atomic. 

\medskip

A sequence $x_0<x_1<\cdots<x_i<\cdots$ in $L$ is said to be a \emph{chain}, and is called a finite chain if the sequence consists of only finitely many elements. In particular, the \emph{length} of $x_0<x_1<\cdots<x_r$ is defined to be $r$. Moreover, the length of the lattice $L$ is defined as the maximum of the lengths of finite chains, if there exists such a maximum; otherwise, $L$ is said to be of infinite length. We say a chain $x_0<x_1<\cdots<x_i<\cdots$ is \emph{maximal} if it cannot be further refined to a chain by inserting more elements; this is the case if and only if $x_j\lessdot x_{j+1}$, for each $j$.

\medskip

We recall that $x \in L$ is \emph{join-irreducible} if, for each $y$ and $z$ in $L$ with $y, z <x$, we have $x\neq y\vee z$. By convention, $\hat{0}$ is not considered as a join-irreducible element. Moreover, $x \in L$ is said to be \emph{completely join-irreducible} if for every subset $S\subseteq \{y \in L \mid y<x\}$, we have $x \neq \bigvee_{y\in S} y$. 
Note that $x$ is completely join irreducible if and only if there is an element $x_*$ such that $y<x$ if and only if $y<x_*$. Provided $x$ is join-irreducible but not completely join-irreducible, such an element $x_*$ does not exist. 
Let $\JI(L)$ denote the set of all join-irreducible elements in $L$, and $\JI^c(L)$ denotes the subset of $\JI(L)$ consisting of completely join-irreducible elements. 
The notion of (completely) meet-irreducible elements, and consequently $\MI(L)$ and $\MI^c(L)$, are defined dually. Note that, by convention, $\hat{1}$ does not belong to $\MI(L)$. In particulr, $m$ belongs to $\MI^c(L)$ if and only if there is an element $m^*$ such that $m<y$ if and only if $m^*<y$.

\medskip

A lattice $L$ is called \emph{semidistributive} if, for all $x, y, z \in L$, if $x\vee y = x\vee z$, then $ x \vee (y \wedge z)=x \vee y$, and, furthermore, $x\wedge y = x\wedge z$ implies $x \wedge (y \vee z)=x \wedge y$. 
Consequently, for any cover relation $x\lessdot y$, the set $\{a \in L \,|\, x\vee a=y$\} has a unique minimal element.  
A complete lattice $L$ is \emph{completely semidistributive} if, for every
$x \in L$ and each $S \subseteq L$, the following hold:
\begin{itemize}
    \item If $x \vee y = x \vee z$ for all $y, z \in S$, then $x \vee (\bigwedge S)= x\vee y$, for all $y\in S$.
    \item If $x \wedge y = x \wedge z$ for all $y, z \in S$, then $x \wedge (\bigvee S)= x\wedge y$, for all $y\in S$.
\end{itemize}

A lattice $L$ is said to be \emph{spatial} if every element $x$ in $L$ can be written as a (possibly
infinite) join of completely join-irreducible elements; that is, $x= \bigvee_{j\leq x} j$, where $j\in \JI^c(L)$. The lattice $L$ is called \emph{bi-spatial} if it is spatial and the dual condition holds; namely, every $x \in L$ can be written as a join of completely join-irreducible elements, as well as a meet of completely meet-irreducible elements. 

\medskip

An important subfamily of bi-spatial lattices consists of those which admit particular bijections between $\JI^c(L)$ and $\MI^c(L)$. More specifically, a bi-spatial lattice $L$ is said to be a \emph{$\kappa$-lattice} if there are inverse bijections $\kappa: \JI^c(L) \to \MI^c(L)$ and $\kappa^{-1}: \MI^c(L) \to \JI^c(L)$, respectively defined by $\kappa(j):=\max \{x\in L \,|\, x\wedge j= j_*\}$ and $\kappa^{-1}(m):=\min \{y\in L \,|\, y\vee m= m^*\}$. 
Provided $L$ is a finite lattice, then $L$ is semidistributive if and only if it is a $\kappa$-lattice (for instance, see \cite[Section 2.5]{RST}). 
A $\kappa$-lattice $L$ is called \emph{well-separated} if whenever $x\not\leq y$, there exists $j \in \JI^c(L)$ such that $j \leq x$ and $y\leq \kappa(j)$. 
The following theorem puts the above definitions into a better perspective. Meanwhile, we emphasize that for an arbitrary complete lattice, the condition of being a well-separated $\kappa$-lattice neither implies nor is implied by being completely semidistributive (see Remark \ref{Rem: comparison between our two implications}). 
For a detailed comparison of the above lattice-theoretical conditions, see \cite[Sections 3.1 and 3.3]{RST}.

\begin{theorem}[{\cite[Theorem 3.1]{RST}}]\label{Thm:RST implications}
For a complete lattice $L$ that is completely semidistributive, the following hold:
\begin{enumerate}
    \item If $L$ is weakly atomic, then $L$ is bi-spatial.
    \item $L$ is bi-spatial if and only if $L$ is a $\kappa$-lattice.
\end{enumerate}
\end{theorem}

For an arbitrary lattice $L$, recall that an element $x$ in $L$ is said to be \emph{left modular} if for every $y \leq z$ in $L$, we have the equality $(y\vee x)\wedge z = y\vee (x\wedge z)$. Observe that, for a pair of elements $y\leq z$ in $L$, the inequality $(y\vee x)\wedge z \geq y \vee (x\wedge z)$ holds for all $x\in L$; often known as \emph{modular inequality}. Therefore, to verify whether $x$ in $L$ is left modular, one only needs to verify the non-trivial inequality. Let $\LM(L)$ denote the set of all left modular elements of $L$. 

\medskip

For finite lattices, the study of left modular elements, as well as the length of a lattice and its maximal chains, have inspired many works. This has ultimately led to various generalizations of (finite) distributive lattices. 
For instance, in \cite{St2} the author introduced \emph{supersolvable} lattices as graded lattices which have a maximal chain consisting entirely of left modular elements (see also \cite{St1}). 
Moreover, in \cite{LS} left modularity is treated as a generalization of the more restrictive notion of modular elements, where the authors gave some alternative characterizations of left modular elements \cite[Theorem 1.4]{LS}, and generalize some earlier results of \cite{BS}. An arbitrary finite lattice (not necessarily graded) $L$ is said to be \emph{left modular} if it possesses a maximal chain consisting entirely of left modular elements. Every finite distributive lattice is obviously left modular.

\medskip

Another generalization of (finite) distributive lattices is given in terms of the length of $L$ compared to $|{\JI(L)}|$ and $|{\MI(L)}|$. Following \cite{Ma}, a lattice $L$ of length $m$ is said to be \emph{extremal} if $|{\JI(L)}|=|{\MI(L)}|=m$. We remark that the inequalities $|{\JI(L)}|\geq m$ and $|{\MI(L)}|\geq m$ always hold. 
For a detailed comparison between the notions of left modularity and extremality (and some related notions) in the setting of finite lattices, we refer to \cite{Mu}. 
We state the following result, which clarifies the connections between left modularity and extremality in the setting of finite lattices. 
For more recent studies on finite extremal semidistributive lattices, see \cite{Se} and the references therein.

\begin{theorem}[\cite{Mu,TW}]\label{Thm: trim=extremal=left modular}
Let $L$ be a finite semidistributive lattice. Then $L$ is left modular if and only if $L$ is extremal.
\end{theorem}

We finally recall that a finite lattice $L$ is called \emph{trim} if $L$ is both extremal and left modular. This family also gives another generalization of finite distributive lattices; an analogue of distributivity for ungraded lattices (for details, see \cite{Th}). 
We note that a trim lattice is not necessarily semidistributive. 
That being the case, the preceding theorem also gives a characterization of semidistributive trim lattice. More specifically, a finite semidistributive lattice $L$ is trim if and only if $L$ is left modular, if and only if $L$ is extremal.

\subsection{Torsion classes}\label{Subsec: Torsion Classes}
In the following, by $A$ we always denote a basic connected finite dimensional $k$-algebra which satisfies the assumptions from Section \ref{Sec: Summary of Main Results}. Moreover, we only work with subcategories $\mathcal{C}$ of $\modu A$ which are full and closed under isomorphism classes. Set $\mathcal{C}^\perp:=\{X \in \modu A\,|\, \Hom_A(C,X)=0, \text{ for each } C \in \mathcal{C}\}$. Moreover, ${}^\perp \mathcal{C}$ is defined dually. 
A subcategory $\mathcal{C}$ of $\modu A$ is \emph{extension-closed} if for every short exact sequence $0\rightarrow L \rightarrow M \rightarrow N \rightarrow 0$ in $\modu A$, provided $L$ and $N$ belong to $\mathcal{C}$, then $M \in \mathcal{C}$. 

For two subcategories $\mathcal{T}$ and $\mathcal{F}$ of $\modu A$, we say $(\mathcal{T},\mathcal{F})$ is a \emph{torsion pair} in $\modu A$ if the following conditions hold:
\begin{enumerate}
    \item $\mathcal{T}$ and $\mathcal{F}$ have only $0$ in common.
    \item $\mathcal{T}$ is closed under taking factors and $\mathcal{F}$ is closed under taking submodules.
    \item For each $M$ in $\modu A$, there exists a unique submodule $t(M)$ of $M$ such that $0 \rightarrow t(M) \rightarrow M \rightarrow M/t(M)\rightarrow 0$ is exact with $t(M) \in \mathcal{T}$ and $M/t(M) \in \mathcal{F}$.
\end{enumerate}

For a torsion pair $(\mathcal{T},\mathcal{F})$ in $\modu A$, we respectively call $\mathcal{T}$ and $\mathcal{F}$ as the \emph{torsion class} and the \emph{torsion-free class}. That being the case, each module in $\mathcal{T}$ is called a torsion module, and those modules belonging to $\mathcal{F}$ are known as torsion-free modules. Observe that $\mathcal{T}$ and $\mathcal{F}$ are extension-closed, and each one uniquely determines the pair. In particular,  $\mathcal{F}=\mathcal{T}^\perp$ and $\mathcal{T}= {}^\perp\mathcal{F}$ (for details, see \cite[Chapter VI.1]{ASS}). 
Obviously, $(0,\modu A)$ and $(\modu A, 0)$ form torsion pairs. Thus, we say $(\mathcal{T},\mathcal{F})$ is a non-trivial torsion pair if $\mathcal{T}\neq 0$ and $\mathcal{T}\neq \modu A$. 

\medskip

If $\mathcal{C}$ is a subcategory of $\modu A$, the smallest extension-closed subcategory of $\modu A$ that contains $\mathcal{C}$ is denoted by $\Filt\mathcal{C}$, and it consists of all modules in $\modu A$ which admit a filtration by the objects in $\mathcal{C}$. Moreover, let $\Fac\mathcal{C}$ (resp. $\Sub\mathcal{C}$) denote the full subcategory of $\modu A$ consisting of all $N$ in $\modu A$ that are quotients (resp. submodules) of a finite direct sum of objects in $\mathcal{C}$. 
Define 
$\mathcal{T(\mathcal{C}}):=\Filt(\Fac\mathcal{C})$ and $\mathcal{F(\mathcal{C}}):=\Filt(\Sub\mathcal{C})$.
Therefore, every $M$ in $\mathcal{T}(\mathcal{C})$ has a filtration $0= M_0 \subseteq M_1 \subseteq \cdots \subseteq M_{d-1} \subseteq M_d=M$ such that for each $1 \leq i \leq d$, there exists an epimorphism $\psi_i:C_i\twoheadrightarrow M_{i}/M_{i-1}$ for some $C_i$ that is a finite direct sum of objects from $\mathcal{C}$. The following lemma is well known. In particular, it implies that each torsion class is determined by the indecomposable modules it contains.

 \begin{lemma}\label{Smallest Torsion Class}
Let $\mathcal{C}$ be a subcategory of $\modu A$. The smallest torsion class in $\modu A$ that contains $\mathcal{C}$ is $\mathcal{T}(\mathcal{C})$, and the smallest torsion-free class containing $\mathcal{C}$ is $\mathcal{F}(\mathcal{C})$.
\end{lemma}

Let $\tors A$ and $\torf A$, respectively, denote the set of all torsion classes and torsion-free classes in $\modu A$. Both sets form lattices with respect to the inclusion-order: For a subset $\{\mathcal{T}_i\}_{i \in I}$ in $ \tors A$ (respectively, $\{\mathcal{F}_i\}_{i \in I}$ in $\torf A$), the meet is given by $\bigwedge_{i\in I}\mathcal{T}_i = \bigcap_{i \in I}\mathcal{T}_i$ (respectively, $\bigwedge_{i\in I}\mathcal{F}_i = \bigcap_{i \in I}\mathcal{F}_i$). The join $\bigvee_{i\in I}\mathcal{T}_i$ is defined as the intersection of all torsion classes $\mathcal{T}\in \tors A$ which contain $\bigcup_{i\in I}\mathcal{T}_i$. The join $\bigvee_{i\in I}\mathcal{F}_i$ is defined analogously. Moreover, there is an anti-isomorphism of lattices $\tors A \to \torf A$, where $\mathcal{T}\in \tors A$ is mapped to $\mathcal{T}^{\bot} \in \torf A$, while, in the opposite direction, $\mathcal{F}$ is sent to $^{\bot}\mathcal{F}$. 
We also recall that a torsion class $\mathcal{T}$ is \emph{functorially finite} if $\mathcal{T}=\Fac(M)$, for some $M$ in $\modu A$.

\medskip

It is known that $\tors A$ (and $\torf A$) is a lattice that possesses many interesting features. Before we list some of these important properties in the next theorem, let us recall that $\brick A$ denotes the set of (isomorphism classes of) bricks in $\modu A$, which are those modules whose endomorphism algebra is a division algebra. Then, $A$ is called a \emph{brick-finite} algebra provided that $\modu A$ admits only finitely many bricks, up to isomorphism.

\medskip

We now list some of the important properties of $\tors A$, which also holds for $\torf A$. This theorem lists several results proved by various authors, including \cite{BCZ,DIJ, DI+,GM,RST}. Here we choose a more lattice theoretical presentation and for details, we refer to the aforementioned papers and the references therein. 

\newpage

\begin{theorem}\label{Thm: properties of tors(A)}
Let $A$ be an algebra. Then, 
\begin{enumerate}
    \item $\tors A$ is a complete lattice which is completely semidistributive lattice and weakly atomic. 
    \item There is a brick labelling of $\Hasse(\tors A)$, meaning that, each arrow in $\Hasse(\tors A)$ is uniquely labeled by an element of $\brick A$.
    \item If $\mathcal{T}_l \to \mathcal{T}_{l-1} \to \cdots \to \mathcal{T}_0$ is a path in $\Hasse(\tors A)$ and brick $X_i$ labels $\mathcal{T}_i \to \mathcal{T}_{i-1}$ for each $i \in \{1,\ldots,l\}$, then $\Hom_A(X_i,X_j)=0$ for any $i<j$.
    \item The three sets $\brick A$, $\JI^c(\tors A)$, and $\MI^c(\tors A)$ are in bijection.
    \item $\tors A$ is a finite lattice if and only if $A$ is brick-finite.
\end{enumerate}
\end{theorem}

Let us remark that a single element of $\brick A$ may appear as the label of more than one arrow in $\Hasse(\tors A)$, even if $A$ is brick-finite.

\bigskip

\subsection{Brick-directed algebras}\label{Subsection:Brick-directed Algebras} 
Inspired by the lattice-theoretical approach to the study of representation theory of algebras, in \cite{AI+} we introduced modern analogues and novel generalizations of some important classical notions. Here we only present a brief summary of the results related to the content of this work, and for details and proofs, we refer to the aforementioned paper and references therein.

\medskip

As a brick-analogue of the classical splitting torsion pairs, we say that a torsion pair $(\mathcal{T},\mathcal{F})$ in $\modu A$ is \emph{brick-splitting} if, for every $M \in \brick A$, either $M \in \mathcal{T}$ or $M \in \mathcal{F}$. For such a pair, $\mathcal{T}$ is said to be a brick-splitting torsion class and $\mathcal{F}$ is called a brick-splitting torsion-free class. 
We remark that brick-splitting torsion pairs give a non-trivial generalization of splitting torsion pairs; More explicitly, for an algebra $A$, there can exist torsion pairs $(\mathcal{T},\mathcal{F})$ which are brick-splitting, but not splitting in the classical sense (see \cite[Example 3.12]{AI+}). 

\medskip


To better present the significance of brick-splitting torsion pairs in our work, particularly through the lens of a directedness property, we recall that a \emph{path} in $\modu A$ is a sequence $X_0 \xrightarrow{f_1} X_1 \xrightarrow{f_2}  \cdots X_{t-1} \xrightarrow{f_t}  X_{t}$, where $t$ is a positive integer, $X_0, \cdots , X_t$ are indecomposables, and for $1\leq i \leq t$, $f_i:X_{i-1}\rightarrow X_i$ is a non-zero non-invertible morphism in $\modu A$. Such a path is a \emph{cycle} if $X_0$ and $X_t$ are isomorphic. Note that in a path, some of the morphisms may belong to $\rad^{\infty}(A)$. In particular, a cycle $X_0\xrightarrow{f_1} X_1\xrightarrow{f_2}  \cdots \xrightarrow{f_{m-1}}  X_{m-1}\xrightarrow{f_m}  X_m=X_0$ in $\modu A$ is called a \emph{brick-cycle} if $X_i \in \brick A$, for all $0\leq i\leq m$. Following \cite{AI+}, we say $A$ is \emph{brick-directed} algebra if there exists no brick-cycle in $\modu A$. 
Observe that each representation-directed algebra is brick-directed and representation-finite, but there are a multitude of brick-directed algebras that are neither representation-directed nor representation-finite (see \cite[Section 4.3]{AI+}). 

\medskip

Before considering arbitrary algebras, we recall the following results on hereditary algebras (i.e., the path algebra of acyclic quivers) to put the new notions into perspective. As observed in \cite[Corollary 4.8]{AI+}, strictly wild algebras are never brick-directed. Meanwhile, for a hereditary algebra $A=kQ$, it is  well known that $A$ is wild if and only it is strictly wild (This equivalence, however, is not true for arbitrary algebras.). This partly explains why in the first two statements of the following proposition, only tame hereditary algebras occur. For more details, see \cite[Sections 3 \& 4]{AI+}, and references therein.

\newpage

\begin{proposition}{\cite[Prop. 3.10 \& 3.11 \& 4.9]{AI+}}
Let $A$ be a connected path algebra, that is $A=kQ$. Then,
\begin{enumerate}
    \item $A$ is brick-directed if and only if $Q$ is a Dynkin or the Kronecker quiver.
    \item If $A$ is of tame type, then a torsion pair $(\mathcal{T},\mathcal{F})$ is splitting if and only if it is brick-splitting.
    \item Let $(\mathcal{T},\mathcal{F})$ be a functorially finite torsion pair in $\modu A$. Then, $(\mathcal{T},\mathcal{F})$ is splitting if and only if it is brick-splitting.
\end{enumerate}
\end{proposition}

In the following theorem, where we treat arbitrary algebras, we summarize some of the main results of \cite{AI+} and highlight the linkages between the notion of brick-directedness and the lattice-theoretical properties investigated in this work. 
We recall that for a connected algebra $A$, by $Q^b(A)$ we denote the \emph{brick quiver} of $A$, where the vertices of $Q^b(A)$ are in bijection with elements of $\brick A$, and for each pair of distinct (hence non-isomorphic) modules $X$ and $Y$ in $\brick A$, we put an arrow from $X$ to $Y$ if $\Hom_A(X,Y)\neq 0$.

\begin{theorem}{\cite[Theorems 1.4 \& 1.7 \& 4.1]{AI+}}\label{Thm: Characterization of brick-directed via maximal chain of brick-splitting pairs}
The following are equivalent:
\begin{enumerate}
    \item $A$ is brick-directed;
    \item There exists a maximal chain $\{\mathcal{T}_i\}_{i\in I}$ in $\tors A$ such that every torsion class $\mathcal{T}_i$ is brick-splitting.
    \item $Q^b(A)$ is acyclic.
\end{enumerate}
Moreover, if $A$ is brick-finite, then $A$ is brick-directed if and only if $\tors A$ is a left modular (equivalently, extremal) lattice. That being the case, $\LM(\tors A)$ consists of all brick-splitting torsion classes and it forms a distributive sublattice of $\tors A$. 
\end{theorem}

\bigskip

\section{Left modularity and extremality for well-separated $\kappa$-lattices}\label{Section: Left modularity and extremality for well-separated kappa-lattices}

Recall that for a lattice $(L,\leq )$, an element $x$ of $L$ is called \emph{left modular} if for any $y\leq z$, we have $(y\vee x)\wedge z=y\vee(x\wedge z)$. A finite lattice is said to be left modular if it has an inclusion-maximal chain of left modular elements. We also recall that a finite lattice is said to be extremal if the length of its longest chain equals the number of join-irreducibles and the number of meet-irreducibles. In this section, we aim to generalize the above notions to the setting of arbitrary lattices. 

\subsection{Generalization of left modularity and extremality}\label{Subsection: Generalization of left modularity and extremality for kappa-lattices}

Considering that the definition of left modularity is with respect to the property of the elements appearing on a maximal chain, the same definition can be adopted for arbitrary lattices. 

\begin{definition}\label{Def: Left modularity for arbitrary lattices}
A (possibly-infinite) lattice $L$ is called \emph{left modular} if it has an inclusion-maximal chain of left modular elements.
\end{definition}

As for extremality, however, the generalization is more technical; that is, one cannot simply adopt the same definition that applies to finite lattices. This is partly because, for infinite lattice, the length of a maximal chain and the number of join-irreducibles and meet-irreducibles are not well-defined. In that regard, we give the following definition.

\begin{definition}\label{Def: Extremal for arbitrary lattices}
Let $L$ be a lattice, possibly infinite, and $\mathcal{X}$ be a maximal chain in $L$. Then $\mathcal{X}$ is \emph{$\lambda$-extremal} if there exists a bijection $\lambda: {\JI^c}(L)\to {\MI^c}(L)$ such that for every $x \in \mathcal{X}$, if some $j \in {\JI^c}(L)$ satisfies $j\not\leq x$, then $x\leq \lambda(j)$; and furthermore, if $x\lessdot y$ for two elements $x,y$ of $\mathcal{X}$, then
\begin{enumerate}
    \item there is a unique $j \in {\JI^c}(L)$ with $j \leq y$ but $j \not \leq x$;
    \item there is a unique $m \in {\MI^c}(L)$ with  $x \leq m$ but $y \not \leq m$;
\end{enumerate}
and $\lambda(j)=m$. The lattice $L$ is called \emph{extremal} if it admits a $\lambda$-extremal chain for some bijection $\lambda$. When $\lambda$ is clear from the context, we often suppress it. In particular, if $L$ is a $\kappa$-lattice, unless otherwise stated, we assume that $\lambda=\kappa$.
\end{definition}

\begin{remark}\label{Rem: Technicality in extremality}
In the above definition, in order to guarantee the uniqueness of the $j\in \JI^c(L)$ and $m\in \MI^c(L)$ in parts (1) and (2), observe that the conditions should be stated in the extremal chain $\mathcal{X}$ in $L$, and not the entire lattice $L$. 
For instance, for the lattice $L$ in Figure \ref{Fig. for A_2}, we have a cover relation $x \lessdot \hat{1}$ and two (completely) join-irreducible elements, namely $j=y$ and $j=z$, with the property that $j \not\le x$ and $j \le \hat{1}$. 
\end{remark}

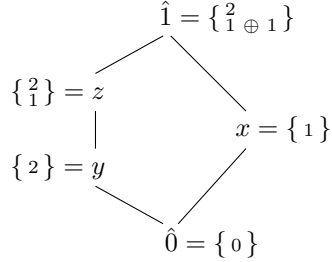
\begin{figure}[h]
\begin{center}
    \begin{tikzpicture}
\node at (0.75,4) {$\hat{1}= \{{\begin{smallmatrix} 2 & & \\1 &\oplus & 1 \end{smallmatrix}}\}$};
    \draw [-] (-0.1,3.75)--(-1,3.25); 
\node at (-1.5,3) {$\{{\begin{smallmatrix} 2\\1\end{smallmatrix}}\}=z$};
    \draw [-] (-1,2.75)--(-1,2.25); 
\node at (-1.5,2) {$\{{\begin{smallmatrix} 2\end{smallmatrix}}\}=y$};
    \draw [-] (-1,1.75)--(-0.1,1.25); 
\node at (0.55,1) {$\hat{0}=\{{\begin{smallmatrix} 0\end{smallmatrix}}\}$};;
    \draw [-] (1,2.25)--(0.1,1.25); 
\node at (1.5,2.5) {$x=\{{\begin{smallmatrix} 1\end{smallmatrix}}\}$};
    \draw [-] (0,3.75)--(1,2.75); 
    \end{tikzpicture}
\end{center}
\caption{Lattice of torsion classes of path algebra of $Q: 2\to 1$. Vertices are specified by $\tau$-rigid modules generating torsion classes.}\label{Fig. for A_2}
\end{figure}

To put our more general notion of extremality in Definition \ref{Def: Extremal for arbitrary lattices} into perspective, we compare it with the classical notion of extremality, only given for finite lattices (see Section \ref{Subsec: Lattice Theory}). 
In particular, the following lemma shows that our notion of extremality coincides with the classical one if we restrict to finite lattices. 

\begin{lemma}\label{Lem: generalized extremal for finite lattices} 
A finite lattice $L$ is extremal in the sense of Definition \ref{Def: Extremal for arbitrary lattices} if and only if it is extremal in the usual sense. \end{lemma}

\begin{proof} Let $L$ be a finite lattice $L$ which satisfies our definition of extremality in Definition \ref{Def: Extremal for arbitrary lattices}. 
Let $\hat 0 =x_0 \lessdot x_1 \lessdot \dots \lessdot x_r=\hat 1$ be the elements of a chain $\mathcal{X}$ in $L$ of maximal length. By induction, the number of join-irreducibles below $x_i$ is $i$, and so we see that the number of join-irreducibles equals the length of the chain $\mathcal{X}$. Thus, the number of join-irreducibles is no more than the length of the longest chain, but the number of join-irreducibles cannot be less than the length of the longest chain, so they are equal. Thanks the the bijection $\lambda: {\JI^c}(L)\to {\MI^c}(L)$, we know that there are also the same number of meet-irreducibles. This implies $L$ is extremal in the usual sense; that is, length of $L$ equals to $|{\JI(L)}|=|{\MI(L)}|$.

For the converse, let $L$ be a finite lattice $L$ that satisfies the usual definition of extremality. Assume $\hat 0 =x_0 \lessdot x_1 \lessdot \dots \lessdot x_r=\hat 1$ is a chain of maximal length, and therefore it is of length $|{\JI(L)}|=|{\MI(L)}|$. Suppose $\mathcal{X}$ denotes the elements of this chain. By extremality, below $x_{i+1}$ there is exactly one more join-irreducible comparing to the number of the join-irreducible elements below $x_i$. Similarly, there is exactly one less meet-irreducible above $x_{i+1}$ than above $x_i$. Using this observation, define $\lambda: {\JI^c}(L)\to {\MI^c}(L)$ to match up each join-irreducible to the corresponding meet-irreducible. Therefore, $L$ is extremal in the sense of Definition \ref{Def: Extremal for arbitrary lattices}.
\end{proof} 

\subsection{Equivalence of left modularity and extremality}\label{Subsection: Equivalence of left modularity and extremality for kappa-lattices}
In this subsection, we state two propositions which, together, verify the the first assertion of Theorem \ref{Thm: Equivalences for infinite lattices}. More explicitly, for a complete well-separated $\kappa$-lattice $L$, we show that $L$ is left modular if and only if $L$ is extremal (in the sense of Definitions \ref{Def: Left modularity for arbitrary lattices} and \ref{Def: Extremal for arbitrary lattices}).

\begin{proposition}\label{Prop:For well-separated kappa-lattices, left modular implies extremal}
Let $L$ be a left modular, well-separated $\kappa$-lattice. Then $L$ is extremal. \end{proposition}

\begin{proof}
Let $\mathcal{X}$ be an inclusion-maximal chain of left modular elements in $L$, and suppose $x\in \mathcal{X}$. 
Since $L$ is a $\kappa$-lattice, we will take $\kappa: \JI^c (L) \to \MI^c (L)$ to be our bijection. Let $j$ be a completely join-irreducible with $j\not\leq x$. Then we have $(x\wedge j)\vee j_*=j_*$. By left modularity of $x$, we have
  $(x\vee j_*)\wedge j=j_*$. Let $m=\kappa(j)$. By the definition of $\kappa$, we have $x\vee j_*\leq m$. So $x\leq m$. Thus, for every completely join-irreducible $j\not\leq x$, we have $\kappa(j)\geq x$.

Now, we must check that if $x\lessdot y$ are two elements of $\mathcal{X}$, there exists a unique join-irreducible less than or equal to $y$ which is not less than or equal to $x$. The fact that $L$ is well-separated guarantees that there is a join-irreducible $j$ with $j\leq y$ and $m=\kappa(j)\geq x$. Since $\kappa(j)$ is not greater than $j$, $j\not\leq x$. So $j$ is a join-irreducible that satisfies the conditions. We must check that there is no other such join-irreducible element to verify the uniqueness. 

If $j_*$ were not less than $x$, then $x\vee j_*$ would have to be greater than $x$, and it must be less than or equal to $y$, so it would have to be equal to $y$. But then $m\vee y=m\vee x \vee j_*$, and since $m\geq x$ and $m\geq j_*$, we would have $m\vee y=m$. But this is a contradiction. So we must have $j_*\leq x$.

The dual argument shows that $m^*\geq y$. So we have the following picture

$$
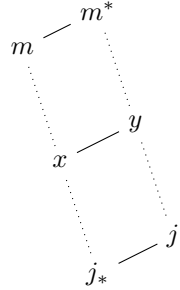
\begin{tikzpicture}
  \node (a) at (0,0) {$m^*$};
  \node (b) at (-1,-.5) {$m$};
  \node (c) at (0.5,-1.5) {$y$};
  \node (d) at (-.5,-2) {$x$};
  \node (e) at (1,-3) {$j$};
  \node (f) at (0,-3.5) {$j_*$};
  \draw (a)--(b);
  \draw (c)--(d);
  \draw (e)--(f);
  \draw[dotted] (e) -- (c); \draw [dotted] (c)--(a);
  \draw [dotted] (f)--(d);\draw[dotted] (d)--(b);\end{tikzpicture}$$
where the straight lines denote cover relations, whereas the dotted lines denote order relations which need not be covers.

Now let $j'$ be another complete join-irreducible such that $x\vee j'=y$.
Thus
$m\vee j'=m^*$. We know that the minimum element $z$ satisfying $m\vee z=m^*$ is
$j$. So $j'>j$, and $j'_*\geq j$. Now $j\wedge x<j$, so $j\wedge x \leq j_*$. Thus
$(x \wedge j)\vee j_*=j_*$. On the other hand, since $j'_*\not\leq x$, we must have $j'_*\vee x>x$. But we also see that $j'_*\vee x\leq y$. So $j'_*\vee x=y$.
Thus $(x\vee j'_*)\wedge j'=j'$. This contradicts the left modularity of $x$ and implies the desired uniqueness, which finishes the proof.
\end{proof}

Next, we show the other implication.

\begin{proposition} \label{Prop:For well-separated kappa-lattices, extremal implies left modular}
Let $L$ be an extremal well-separated $\kappa$-lattice. Then $L$ is left modular. \end{proposition}

\begin{proof} Let $\mathcal{X}$ be the inclusion-maximal chain whose existence is guaranteed by extremality in Definition \ref{Def: Extremal for arbitrary lattices}. It suffices to show all elements of $\mathcal{X}$ are left modular. 

For the sake of contradiction, let $x$ be an element of $\mathcal{X}$ such that for some $y\leq z$ in $L$ we have $(x\vee y)\wedge z > (x\wedge z)\vee y$. 
Then, by well-separation property of $L$, there is some join-irreducible $j$ such that $j\leq (x\vee y)\wedge z$ and $m=\kappa(j)\geq (x\wedge z)\vee y$.
So $j\leq (x\vee y)$ and $j\leq z$. 

Since $m\geq (x\wedge z)\vee y$, but $j\not\leq m$, we must have $j\not\leq (x\wedge z)\vee y$. Thus, $j\not\leq (x\wedge z)$. But we already have $j\leq z$. Hence, $j\not\leq x$.
By a dual argument, we show $m\not\geq x$. This violates the extremality of $L$ and implies the desired contradiction.
  \end{proof}

Propositions \ref{Prop:For well-separated kappa-lattices, left modular implies extremal} and \ref{Prop:For well-separated kappa-lattices, extremal implies left modular}, together, imply the important equivalence claimed in the beginning of this subsection, which we state in the following theorem. In particular, this verifies part (1) of Theorem \ref{Thm: Equivalences for infinite lattices}.

\begin{theorem}\label{Thm: for well-sep kappa lattice, left modular equi to extremal}
Let $L$ be a well-separated $\kappa$-lattice. Then, $L$ is left modular if and only if $L$ is extremal.
\end{theorem}

\bigskip

\section{Left modularity and extremality for weakly atomic completely semidistributive lattices}\label{Section: Left modularity and Extremality for weakly atomic lattices}

Here we focus on (complete) weakly atomic completely semidistributive lattices. In particular, in Subsection \ref{Subsection: Left modularity in weakly atomic completely semidistributive lattices}, we study the left modular elements of such lattices. Subsequently, we treat the notion of left modularity for this family and prove Theorem \ref{Thm: left modular elements in infinite lattices}. 
In Subsection \ref{Subsection: Extremality in weakly atomic completely semidistributive lattices}, we investigate the extremality of these lattices and give a less technical realization of extremality in this setting. Moreover, we prove part (2) of Theorem \ref{Thm: Equivalences for infinite lattices}.
In Subsection \ref{Subsection: The labelling quiver}, we introduce the labelling quivers and successor-closed sets, which gives a different approach to the study of left modularity and extremality in this setting, then verify Theorems \ref{Thm: LM(L) and successor-closed sets in Introduction} and \ref{Thm: labelling quiver, linear extensions and extremal chains in Introduction}.

\subsection{Left modularity in weakly atomic completely semidistributive lattices}\label{Subsection: Left modularity in weakly atomic completely semidistributive lattices}

In this section, we begin our treatment of (possibly infinite but complete) weakly atomic completely semidistributive lattices. In particular, we study the left modular elements (and consequently, left modularity) for such lattices.

\medskip

The following lemma is helpful and they are freely used below. 

\begin{lemma}\label{Lemma: cover relation and join-irred}
    If $x \lessdot y$ is a cover relation in $L$, then there is a unique $j \in \JI^c(L)$ with $x \vee j = y$ and $y \wedge \kappa(j)=x$. This element is the label of the cover relation.
\end{lemma}

\begin{proof}
    Existence. Let $j:=j(x \lessdot y)$ denote the completely join-irreducible labelling of that edge. From \cite[Lemma 3.11]{RST}, it follows that $j=\min \{a \in L \mid a \vee x = y\}$, and it has the property that $j_* \le x$. Consequently, $j \vee x = y$.

    We need to check that $x \le \kappa(j)$. Since $x \wedge j = j_* = \kappa(j) \wedge j$, semidistributivity implies $j_* = j \wedge (x \vee \kappa(j))$. By definition of $\kappa(j)$, this yields $x \vee \kappa(j) \leq \kappa(j)$, and therefore $x \leq \kappa(j)$.

    From above, we have $x \leq y \wedge \kappa(j) \leq y$. If $y = y \wedge \kappa(j)$, then $y \le \kappa(j)$, leading to $j \leq \kappa(j)$, a contradiction. Consequently, $x \lessdot y$ implies $x = \kappa(j)\wedge y$.

    Uniqueness. Assume that $j'$ is another completely join irreducible element with $x \vee j' = y$ and $y \wedge \kappa(j')=x$. By the defining property of $j:= j(x \lessdot y)$, it follows that $j \le j'$, hence $\kappa(j) \ge \kappa(j')$. Since $j' \le y$, we get $j'_*=\kappa(j') \wedge j' \le \kappa(j') \wedge y = x$. Similarly, $\kappa(j') \ge x$ gives $\kappa(j')^*=\kappa(j') \vee j' \ge x \vee j' = y$. 
    Since $\kappa(j')$ is completely meet irreducible, if $\kappa(j') < \kappa(j)$, then $y \le \kappa(j')^* \le \kappa(j)$, a contradiction.    
\end{proof}

Before showing another handy lemma, let us emphasize that in the previous proof, semidistributivity is essential.

\begin{lemma}\label{Lem: join-irr cover relation}
    If there is $j \in \JI^c(L)$ with $x \vee j = y$ and $y \wedge \kappa(j)=x$, then $x \lessdot y$.
\end{lemma}

\begin{proof}
Clearly, we have $x \le y$. If $x = y$, then $j \le x \le \kappa(j)$, a contradiction. Let $z$ be such that $x < z < y$. Because $j \le y$, we get $j_*=j \wedge \kappa(j) \le y \wedge \kappa(j) = x$. As $j_* \le x \le z$, we have $z \wedge j = \kappa(j) \wedge j = j_*$. By semidistributivity, we get $j_* = j \wedge (z \vee \kappa(j))$, so that $z \vee \kappa(j) \le \kappa(j)$ and $z \le \kappa(j)$. Then $z \le y \wedge \kappa(j) = x$, which is a contradiction.
\end{proof}

\medskip

The following proposition was originally inspired by the setting of the lattice of torsion classes and in terms of brick-splitting torsion pairs (see \cite[Prop. 3.4]{AI+}). More specifically, in the more general setting treated in this section, we give a characterization of left modular elements in complete lattices which are weakly atomic completely semidistributive. 
We recall that for every such lattice $L$, there is a bijection $\kappa:{\JI^c}(L)\to {\MI^c}(L)$ such that $m:=\kappa(j)=\max \{x\in L \,|\, x\wedge j= j_*\}$ which satisfies $m\wedge j=j_*$ and $m\vee j=m^*$ (see Section \ref{Subsec: Lattice Theory} and references therein). 

\begin{proposition}\label{Prop: Characterization of left modular}
Let $L$ be a weakly atomic and completely semidistributive lattice.
Then, for $t\in L$, the following conditions are equivalent:
\begin{enumerate}
    \item $t$ is left modular.
    \item For each $j\in{\JI^c}(L)$, we have $j\le t$ or $t\le \kappa(j)$, but not both.
\end{enumerate}
\end{proposition}

\begin{proof} 
Before we show the implications, observe that, by Theorem \ref{Thm:RST implications}, the lattice $L$ is bi-spatial (equivalently, a $\kappa$-lattice).

(1)$\Rightarrow$(2) Assume $j\ {\not\leq}\ t$, then $j>t\wedge j$ and hence $j_*\ge t\wedge j$. Since $t$ is left modular, we have
\[(j_*\vee t)\wedge j=j_*\vee(t\wedge j)=j_*.\] 
Thus $j_*\vee t\ {\not\geq}\ j$. By definition of $\kappa$, we have $\kappa(j) \ge j_*\vee t\ge t$, as desired.

(2)$\Rightarrow$(1) 
By the proof of \cite[Lemma 2.8]{TW} and weakly atomic property of $L$, in order to prove that $t$ is left modular, it suffices to prove that for a cover relation $z\lessdot y$, we have 
$$(*):\quad (y\wedge t)\vee z= y \wedge (t\vee z).$$
Observe that the containment from left to right is clear. Let $j\in{\JI^c}(L)$ be a join-irreducible labelling of $z\lessdot y$, and $m:=\kappa(j)\in{\MI^c}(L)$.
Then we have $j\vee z=y$ and $m\wedge y=z$ (see Lemma \ref{Lemma: cover relation and join-irred}). By assumption, either $j \leq t$ or $t \leq m$. If $j \leq t$, we see that
$$(y\wedge t)\vee z\geq j\vee z= y\geq y \wedge (t\vee z).$$
Assume now that $t \leq m$. 
Since $t\vee z\leq m$, we have
$$y \wedge (t\vee z) \leq y\wedge m = z\leq(y\wedge t) \vee z.$$
Therefore, $(*)$ also holds in this case.
\end{proof}

We derive some interesting consequences of the above proposition, and moreover compare our characterization of the left modular elements in our setting with some earlier results over finite lattices. First, we show the following statement.

\begin{proposition}\label{Prop: left modular and cover-relation equivalences}
Let $L$ be a weakly atomic completely semidistributive lattice, $y \lessdot z$ a cover relation labeled by $j \in \JI^c(L)$, and 
$m := \kappa(j)$. 
Then, for each $t \in L$:
\begin{enumerate}
    \item $j \leq t$ implies $t \vee y = t \vee z$.
    \item $t \leq m$ implies $t \wedge y = t \wedge z$.
    \item $j \leq t$ and $t \leq m$ cannot both hold.
\end{enumerate}
The converses of $(1)$, $(2)$ and $(3)$ hold, provided that $t\in L$ is a left modular element.
Consequently, the following are equivalent:

\begin{itemize}
    \item [(i)] The element $t$ is left modular in $L$.
    \item [(ii)] For any cover relation $y\lessdot z$ in $L$, we have $t \vee y = t \vee z$ or $t \wedge y = t \wedge z$ but not both.
\end{itemize}
\end{proposition}

\begin{proof}
We first note that, from the assumptions, we get $j \vee y = z$ and $j_* \leq y$. Moreover, $m \wedge z = y$ holds. These are captured in the following picture, where the straight lines denote cover relations, and the dotted lines denote the order relations (which are not necessarily cover relations).
$$
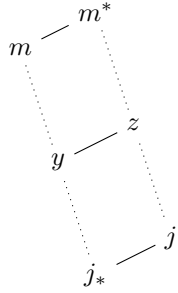
\begin{tikzpicture}
  \node (a) at (0,0) {$m^*$};
  \node (b) at (-1,-.5) {$m$};
  \node (c) at (0.5,-1.5) {$z$};
  \node (d) at (-.5,-2) {$y$};
  \node (e) at (1,-3) {$j$};
  \node (f) at (0,-3.5) {$j_*$};
  \draw (a)--(b);
  \draw (c)--(d);
  \draw (e)--(f);
  \draw[dotted] (e) -- (c); \draw [dotted] (c)--(a);
  \draw [dotted] (f)--(d);\draw[dotted] (d)--(b);\end{tikzpicture}$$

To verify $(1)$, observe that if $j \leq t$, then $z = j \vee y \leq t \vee y$. Hence, $t \vee z \leq t \vee y$. Since $y \leq z$ gives
$t \vee y \leq t \vee z$, we have the desired equality $t \vee y = t \vee z$.

For $(2)$, assume that $t \leq m$. Using this and $y = m \wedge z$, we get $t \wedge z \leq m \wedge z = y$. This yields $t \wedge z \leq t \wedge y$. From $y \leq z$, we get $t \wedge y \leq t \wedge z$, thus $t \wedge y = t \wedge z$.

To show $(3)$, it is clear that if both $j \leq t$ and $t \leq m$ hold simultaneously, then $j \leq m = \kappa(j)$, so
$j \wedge \kappa(j) = j > j_*$, which contradicts the fact that $\kappa(j) \wedge j = j_*$.

\medskip

Now we assume $t\in L$ is left modular and verify the converse implications of $(1)$ and $(2)$. Observe that the converse of $(3)$ follows from Proposition \ref{Prop: Characterization of left modular}. 

For the converse of $(1)$, suppose $t \vee y = t \vee z$. Hence, $z\leq t \vee y$. If $j \not\leq t$, then by Prop. \ref{Prop: Characterization of left modular}, we get $t \leq \kappa(j) = m$. 
Since $m \wedge z = y$, we have $y \leq m$. Together with $t \leq m$, we get $t \vee y \leq m$. Consequently, we get  $z\leq t \vee y \leq m$. 
However, $z \leq y \vee t \leq m$ contradicts
$m \wedge z = y \lessdot z$.

For the converse of $(2)$, assume $t \wedge y = t \wedge z$. If $t \not\leq m$, then by Prop. \ref{Prop: Characterization of left modular}, we get $j \leq t$. 
Observe that $j\vee y=z$ implies $j \leq z$. Together with $j\leq t$ and $t \wedge y = t \wedge z$, this implies $j\leq t \wedge y$, hence $j\leq y$. Therefore, $j\vee y=y$, which contradicts $j\vee y=z$.

\medskip
Finally, we note that the equivalence of (i) and (ii) immediately follows from the above implications.
\end{proof}

Thanks to the previous propositions, in the next remark we make a comparison between the characterizations of the left modular elements in our setting and an analogous result for finite lattices.

\begin{remark}\label{Rem: comparison with LS}
We first recall that if $L$ is a finite semidistributive lattice, then $\JI^c(L) = \JI(L)$. 
That being the case, Prop. \ref{Prop: Characterization of left modular} asserts that $t \in L$ is left modular if and only if for each $j \in \JI(L)$, either $j \leq t$ or $t \leq \kappa(j)$ but not both. 
Moreover, by Prop. \ref{Prop: left modular and cover-relation equivalences}, this is further equivalent to the following condition: for each cover $y \lessdot z$, exactly one of 
$t \vee y = t \vee z$ and $t \wedge y = t \wedge z$ holds. Observe that the latter equivalence is exactly condition (iii) of \cite[Theorem~1.4]{LS}, which characterizes the left modular elements in a finite lattice. 
Hence, for finite semidistributive $\kappa$-lattices, the equivalence shown in Prop. \ref{Prop: left modular and cover-relation equivalences} coincides with that of \cite[Theorem~1.4]{LS}. 
However, our results also applies to a subfamily of infinite lattices, namely those which are weakly atomic and completely semidistributive. 
A more accurate comparison between the two mutually exclusive conditions is presented in the following table:
\begin{center}
\begin{tabular}{c|c}
    \emph{\cite[Theorem 1.4(iii)]{LS}} & \emph{Proposition \ref{Prop: Characterization of left modular}(2)} \\
    \hline 
    $t \vee z = t \vee y$ & $j \leq t$ \\
    $t \wedge z = t \wedge y$ & $t \leq \kappa(j)$ \\
\end{tabular}
\end{center}

Finally, let us emphasize that the passage from finite to infinite lattices requires replacing $\JI(L)$ with $\JI^c(L)$, particularly because in an infinite lattice, an arbitrary element of $\JI(L) \setminus \JI^c(L)$ does not necessarily have a unique lower cover, it is not in the domain of $\kappa$, and it does not label any cover relation. 
The $\kappa$-formulation of Prop. \ref{Prop: Characterization of left modular} is therefore the canonical 
generalization of \cite[Theorem~1.4(iii)]{LS} to the infinite weakly atomic 
completely semidistributive setting, whereas the formulation of Prop. \ref{Prop: left modular and cover-relation equivalences} is analogous to that of \cite[Theorem~1.4(iii)]{LS}.
\end{remark}

Recall that, for an arbitrary lattice $L$, by $\LM(L)$ we denote the set of all left modular elements of $L$. 
The following result shows that, in the setting of the previous propositions, $\LM(L)$ inherits some important lattice-theoretical properties. 
In retrospect, the following result can be seen as the generalization of \cite[Prop. 2.6]{TW} and \cite[Cor. 5.3]{AI+} from finite lattices to infinite lattices. Also, compare with \cite[Theorem 1.1]{AI+} and Theorem \ref{Thm: Characterization of brick-directed via maximal chain of brick-splitting pairs}.

\begin{proposition}\label{Prop: LM(L) is distributive sublattice}
Let $L$ be a weakly atomic and completely semidistributive lattice. 
Then $\LM(L)$ is a complete sublattice of $L$ and it is completely distributive.
\end{proposition}

\begin{proof}
Let $S$ be a subset of $\LM(L)$, We only prove that $x:=\bigwedge S$ is left modular again.
Fix $j\in{\JI^c}(L)$ and $m:=\kappa(j)\in{\MI^c}(L)$, and assume $x\ {\not\leq}\ m$. Then any $s\in S$ satisfies $s\ {\not\leq}\ m$. By Proposition \ref{Prop: Characterization of left modular}, we have $j\leq s$ for each $s\in S$, and hence $j\leq x$. Again by Proposition \ref{Prop: Characterization of left modular}, $x$ is left modular.

It remains to prove that $\LM(L)$ is completely distributive.
Let $\mathcal{P}({\JI^c}(L))$ be the power set of ${\JI^c}(L)$, which forms a complete lattice. Define a map
\[f:\LM(L)\to \mathcal{P}({\JI\nolimits^c}(L)),\ x\mapsto f(x):=\{j\in{\JI\nolimits^c}(L)\mid j\le x\}.\]
This is an injective map since $x=\bigvee f(x)$ holds thanks to the weakly atomic property of $L$.

We prove that $\LM(L)$ is a complete sublattice of $\mathcal{P}({\JI^c}(L))$. Then the claim follows since $\mathcal{P}({\JI^c}(L))$ is completely distributive.
Let $S$ be a subset $S$ of ${\JI^c}(L)$. Since  $j\in{\JI^c}(L)$ satisfies $j\le\bigwedge S$ if and only if $j\le s$ for each $s\in S$.
we have
\[f(\bigwedge S)=\bigwedge_{s\in S} f(s).\]
On the other hand, by the same argument, the map
\[g:\LM(L)\to P({\MI\nolimits^c}(L)),\ x\mapsto g(x):=\{m\in{\MI\nolimits^c}(L)\mid x\le m\}\]
satisfies $g(\bigvee S)=\bigwedge_{s\in S} g(s)$.
By Proposition \ref{Prop: Characterization of left modular}, $f(x)=\kappa^{-1}({\MI^c}(L)\setminus g(x))$ holds for each $x\in\LM(L)$. Thus
\begin{align*}
f(\bigvee S)&=\kappa^{-1}({\MI\nolimits^c}(L)\setminus g(\bigvee S))=\kappa^{-1}({\MI\nolimits^c}(L)\setminus \bigwedge_{s\in S}g(s))\\
&=\bigvee_{s\in S}\kappa^{-1}({\MI\nolimits^c}(L)\setminus g(s))=\bigvee_{s\in S}f(s)
\end{align*}
holds, as desired.
\end{proof}

\medskip

\subsection{Extremality in weakly atomic completely semidistributive lattices}\label{Subsection: Extremality in weakly atomic completely semidistributive lattices}

Throughout, $L$ denotes a complete lattice which is weakly atomic and completely semidistributive. In particular, by Theorem \ref{Thm:RST implications}, $L$ is a bi-spatial $\kappa$-lattice.

\begin{definition} \label{Defn extremality and modularity}
For a maximal chain $\mathcal{X}$ in $L$, we define the following notions:
    \begin{enumerate}
        \item We say $\mathcal{X}$ is \emph{left modular} if every element in $\mathcal{X}$ is left modular.
        \item We say $\mathcal{X}$ is \emph{extremal} if for each $j \in \JI^c(L)$, there are $x < y$ in $\mathcal{X}$ with $x \vee j = y$ and $y \wedge \kappa(j) = x$.
    \end{enumerate}
\end{definition}

In the following remark, we compare the above notion of extremality for completely semidistributive lattices with our earlier definition of $\lambda$-extremality given in Section \ref{Section: Left modularity and extremality for well-separated kappa-lattices}, where the latter applies to a more general setting and therefore it was inevitably more technical.

\begin{remark} \label{Remark of defns}
With the same notation as above, using the bijection $\kappa = \lambda$, the definition of extremality in Definition \ref{Def: Extremal for arbitrary lattices} is as follows: \begin{enumerate}
        \item For every $x \in \mathcal{X}$, if $j \in\JI^c(L)$ with $j \not\leq x$, then $x \leq \kappa(j)$.
        \item For $x \lessdot y$ in $\mathcal{X}$, there is a unique $j \in \JI^c(L)$ with $j \not \leq x$ and $j \leq y$, a unique $m \in \MI^c(L)$ with $x \leq m$ and $y \not \leq m$, and $m = \kappa(j)$. 
    \end{enumerate}
\end{remark}

To highlight a technical point in the definition of extremality in our context and compare with the definition given in Section \ref{Section: Left modularity and extremality for well-separated kappa-lattices}, we refer to Remark \ref{Rem: Technicality in extremality}. This comparison is further discussed in the following.

\begin{proposition} \label{Prop: extremality equals left modularity}
    In our setting, the two definitions of extremality are equivalent, under $\lambda = \kappa$.
\end{proposition}

\begin{proof}
    Assume that $L$ is extremal using Definition \ref{Defn extremality and modularity} and let $\mathcal{X}$ be a maximal chain that satisfies that definition. We verify the two conditions in Remark \ref{Remark of defns}. 
    
    (1) Let $x \in \mathcal{X}$ and $j \in \JI^c(L)$ with $j \not\le x$. We must show $x \leq \kappa(j)$. 
    Note that, by Definition \ref{Defn extremality and modularity}, there exists $u \lessdot v$ in $\mathcal{X}$ such that $u \vee j = v$ and $v \wedge \kappa(j)=u$. Consequently, $x \le u \le \kappa(j)$, as desired.
    
    (2) Let $x \lessdot y$ in $\mathcal{X}$. Consider $j:=j(x \lessdot y)$, which satisfies $j \not\le x$ and $j \le y$. Let $j' \in \JI^c(L)$ also satisfy $j' \not\le x$ and $j' \le y$. 
    Again, by Definition \ref{Defn extremality and modularity}, there is $u \lessdot v$ in $\mathcal{X}$ with $u \vee j' = v$ and $v \wedge \kappa(j')=u$. Either $x \le u$ or $u < x$. In the later case, we get $v \le x$. Since  $j' \not \le u$ and $j' \le v$, we get the contradiction $j' \le x$. Hence, we have $x \le u$. Similarly, $v \le y$. The inequalities  $x \le u \lessdot v \le y$ yield $u=x, v=y$ and $j' = j$. 
    The condition for $\MI^c(L)$ can be verified dually. 
    As shown above, the unique $j\in \JI^c(L)$ with $j \not\le x$ and $j \le y$ is the $j$ with $x \vee j = y$ and $y \wedge \kappa(j) = x$. Similarly, the unique $m \in\MI^c(L)$ with $x \le m$ and $y \not\le m$ is the $m$ with $x \vee \kappa^{-1}(m) = y$ and $y \wedge m = x$. So $\kappa(j)=m$.

    Now, assume that $\mathcal{X}$ is a maximal chain in $L$ satisfying the two conditions from Remark \ref{Remark of defns}. Let $j \in \JI^c(L)$.
    We consider $a:= \bigvee_{j \not\le u, u\in \mathcal{X}}u$, which is in $\mathcal{X}$. Similarly, consider $b:= \bigwedge_{j \le v, v \in \mathcal{X}}v$ which is in $\mathcal{X}$. For $u \in \mathcal{X}$ with $j \not\le u$, we have $u \le \kappa(j)$. If $v \in \mathcal{X}$ with $j \le v$, then $u \le v$ (as otherwise, we get $j \le v < u \le \kappa(j)$, a contradiction). This yields that $a \le b$ and $a \le \kappa(j)$.  If $a = b$, then $j \le b=a \le \kappa(j)$, a contradiction. Hence, $a < b$. Any $c \in \mathcal{X}$ with $a < c \le b$ has to be such that $j \le c$ as otherwise, it would violate the definition of $a$. But then we get $c \le b \le c$ by using the definition of $b$. This shows that $a \lessdot b$. Hence, there is a unique $j' \in \JI^c(L)$ with $j' \not\le a, j' \le b$. Clearly, we have $j' = j$ and hence Definition \ref{Defn extremality and modularity} is satisfied.
\end{proof}

Now, we are ready to prove part (2) of Theorem \ref{Thm: Equivalences for infinite lattices}, concerning the equivalence of left modularity and extremality in the setting of completely semidistributive and weakly atomic lattices.

\begin{theorem}\label{Thm: Equiv left modular & Extremal for weakly atomic completely semidistributive lattices}
    Let $L$ be a completely semidistributive and weakly atomic lattice. Then $L$ is left modular if and only if it is extremal.
\end{theorem}

\begin{proof}
    Let $\mathcal{X}$ be a maximal chain that is extremal. First, we prove that every $t \in \mathcal{\mathcal{X}}$ is left modular. Let $j \in \JI^c(L)$. By extremality, there exist $t' < t''$ in $\mathcal{X}$ with $t' \vee j = t''$ and $t'' \wedge \kappa(j) = t'$. In particular, we get $j \le t''$ and $ t' \le \kappa(j)$. If $t \le t'$, then $t \le \kappa(j)$. Otherwise, if $t \not\le t'$, because $t' < t''$ is a cover relation by Lemma \ref{Lem: join-irr cover relation}, then $j \leq t'' \leq t$. Thus, $t$ is left modular by the characterization of Proposition \ref{Prop: Characterization of left modular}.

Now, to show the other implication, let $\mathcal{X}$ be a maximal chain in $L$ with all elements in $\mathcal{X}$ being left modular. Let $j \in \JI^c(L)$. Consider the following elements in $\mathcal{X}$
$$t:= \bigvee_{j \not \le x, x \in \mathcal{X}} x \qquad \text{ and } \qquad t' = \bigwedge_{j \le y, y \in \mathcal{X}} y.$$
Let $x$ be a term of $t$ and $y$ a term of $t'$. Since $j \not \le x$ and $x$ is left modular, we get $x \le \kappa(j)$. If $y \le x$, then $y \le \kappa(j)$ so that $j \le \kappa(j)$, a contradiction. Hence, $x \le y$, which implies $t \le t'$. 
Note that $j \le t'$ and $t \le \kappa(j)$ so that $j \not \le t$. Thus, $t < t'$. If $t \le s < t'$, then $j \not \le s$. But then, $s \le t \le s$, therefore $t = s$. From this, we conclude that $t < t'$ is a cover relation. Observe that $t < t \vee j \le t'$, implying that $t \vee j = t'$. Similarly, $t \le t' \wedge \kappa(j) < t'$ so that $t' \wedge \kappa(j)=t$. This shows that $\mathcal{X}$ is extremal and this completes the proof.
\end{proof}

\begin{remark}\label{Rem: comparison between our two implications}
   Note that if $L$ is a complete semidistributive and weakly atomic lattice, then $L$ is bi-spatial and $\kappa$-lattice; see \cite[Theorem 3.1]{RST}. However, it does not necessarily imply that $L$ is  well-separated. That is, Proposition \ref{Thm: Equiv left modular & Extremal for weakly atomic completely semidistributive lattices} does not follow from our earlier result in Section \ref{Section: Left modularity and extremality for well-separated kappa-lattices}; particularly from Theorem \ref{Thm: for well-sep kappa lattice, left modular equi to extremal}. 
   
   In fact, for a complete lattice $L$, being well-separated $\kappa$-lattice neither implies, nor is it implied, being weakly atomic completely semidistributive (\cite[Examples 3.22 \& 3.25]{RST}). More specifically, neither of the assertions of Theorem \ref{Thm: Equivalences for infinite lattices} implies the other one, and they treat two distinct families of complete lattices.
\end{remark}

\medskip

\subsection{The labelling quiver and successor-closed sets}\label{Subsection: The labelling quiver}
Throughout, $L$ denotes a weakly atomic and completely semidistributive.

\medskip

The \emph{labelling quiver} of $L$, denoted by $Q_{\JI^c(L)}$ or simply $Q_L$, is the quiver whose vertex set is $\JI^c(L)$ and for $i, j \in L$, we draw an arrow $i \to j$ precisely when $i \ne j$ and $i \not \le \kappa(j)$. 
We note that our description of arrows in $Q_L$ is similar to that considered for any $\kappa$-lattice $L$ in \cite{RST}, where the author denoted that by $\rightarrow_L$; see the paragraph before \cite[Theorem 1.4]{RST}. 
As in Subsection \ref{Subsection: Main Results}, a set $S$ of vertices of $Q_L$ is called \emph{successor-closed} if for any arrow $i\to j$ in $Q_L$, if $i \in S$ then $j \in S$. 
In the following, ${\rm succ}(Q_L)$ denotes the set consisting of all successor-closed sets of vertices of $Q_L$. Evidently, ${\rm succ}(Q_L) \subseteq \mathcal{P}(\JI^c(L))$, where $\mathcal{P}(\JI^c(L))$ denotes the power set of $\JI^c(L)$.

\medskip
\begin{remark}\label{Rem: downset vs. succ-closed}
Observe that what we have called successor-closed sets is sometimes known as downsets (see, for instance, \cite[Section 2.1]{RST}). However, we emphasize that our labelling quivers may have cycles; therefore, they may not be necessarily posets. Thus, to avoid any confusion by the terminology, we use the term successor-closed, instead of downset, since the latter is more common in the setting of posets.
\end{remark}

We now prove the following characterization of left modular elements in terms of the successor-closed sets. In particular, we show Theorem \ref{Thm: LM(L) and successor-closed sets in Introduction}.

\begin{theorem}\label{Thm: bijection between LM(L) and successor-closed sets}
    The map
    $\varphi: \LM(L) \to {\rm succ}(Q_L)$ with $\varphi(t) = \{j \in \JI^c(L) \mid j \le t\}$ is a bijection, whose inverse $\psi: {\rm succ}(Q_L) \to \LM(L)$ is given by $\psi(S) = \bigvee S$.
\end{theorem}

\begin{proof}
    Let $t$ be left modular in $L$. First we verify the map is well-defined, by showing that $\varphi(t)$ is successor-closed. Let $i \to j$ be an arrow in $Q_L$, and assume that $i \in \varphi(t)$. Evidently $i \le t$, and by the definition of the labelling quiver, we have $i \not \le \kappa(j)$. By Prop. \ref{Prop: Characterization of left modular}, we have $j \le t$ or $t \le \kappa(j)$. In the latter case, we get $i \le \kappa(j)$, which is absurd. Hence, $j \le t$ and consequently $j \in \varphi(t)$. Thus, $\varphi$ is well-defined.

    Now, let $S \in {\rm succ}(Q_L)$ and put $t := \bigvee S$. We again use Prop. \ref{Prop: Characterization of left modular} to verify that $t$ is left modular. Let $j \in \JI^c(L)$. If $j \in S$, then $j \le t$ and we are done. Hence, we assume $j \not \in S$ and show that $t \leq \kappa(j)$. Since $j \not \in S$, by definition, for any $i \in S$, we have $i \le \kappa(j)$. This yields $t = \bigvee S \le \kappa(j)$, which proves the desired result. 

\medskip
Now we prove that $\varphi$ is a bijection. 
From the definitions, it is obvious that $\psi \circ \varphi$ is the identity. 
Now, observe that for $S \in {\rm succ}(Q_L)$, we have
    $$\varphi (\psi(S)) = \{j \in {\JI}^c(L) \mid j \le \bigvee S\}.$$
    It is clear that $S \subseteq \varphi (\psi(S))$. If $j \in \varphi (\psi(S))$, then $j \le \bigvee S$. We claim $j\in S$. If we assume otherwise, for each $i \in S$, we have $i \le \kappa(j)$; hence, $\bigvee S \le \kappa(j)$. This yields $j \le \kappa(j)$; a contradiction. Hence, we have the reverse inclusion $\varphi (\psi(S)) \subseteq S$.
\end{proof}

\begin{remark}
    The map $\varphi$ can be seen as the co-restriction of the map $f$ from the proof of Proposition \ref{Prop: LM(L) is distributive sublattice}, where we co-restrict $\mathcal{P}(\JI^c(L))$ to ${\rm succ}(Q_L)$.
\end{remark}

The above theorem has the following useful consequence. 

\begin{corollary}\label{Cor: Q_L being acyclic}
Let $(L,\leq)$ be a weakly atomic completely semidistributive lattice. 
If $L$ is extremal, then the labelling quiver $Q_L$ has no oriented cycles.
\end{corollary}
\begin{proof}
For the sake of contradiction, assume $Q_L$ contains an oriented cycle. 
Let $i,j \in \JI^c(L)$ lie on an oriented cycle. Consider $\mathcal{X}$ an extremal chain in $L$. There exist $u,v \in \mathcal{X}$ such that $u \vee i = v$ and $v \wedge i = u$. Similarly, there exist $x,y \in \mathcal{X}$ such that $x \vee j = y$ and $y \wedge j = x$. With no loss of generality, assume that $u < v \le x < y$. Since $v$ is on $\mathcal{X}$, it is left modular. 
Because $i \le v$, and since there is a path in $Q_L$ from $i$ to $j$, we conclude that $j \le v$. However, this leads to $j \le x$, which is the desired contradiction that.
\end{proof} 
 
By the above corollary, if $L$ is extremal, the labelling quiver $Q_L$  is acyclic, and therefore it possesses the structure of a poset. More explicitly, for $x, y \in \JI^c(L)$, we write $x \preceq y$ if there is an oriented path from $y$ to $x$. The following result on the poset $(Q_L, \preceq)$ gives an explicit construction of all extremal chains in $L$. 
In particular, we prove Theorem \ref{Thm: labelling quiver, linear extensions and extremal chains in Introduction}. 
Before showing this result, let us recall that, for a fixed set $S$, and two orders $\le_1$ and $\le_2$ on $S$, we say $\le_2$ \emph{extends} $\le_1$ provided that for $x, y \in S$, we have $x \le_1 y$ implies $x \le_2 y$. Moreover, $\le_2$ is called a \emph{linear extension} of $\le_1$, if $\le_2$ is a total order that extends $\le_1$. We also note that, as mentioned in Remark \ref{Rem: downset vs. succ-closed}, for a poset $(P,\preceq)$, the downsets coincide with the successor-closed sets in $(P,\preceq)$, defined earlier in this subsection.

\begin{theorem}\label{Thm: labelling quiver, linear extensions and extremal chains}
Let $(L,\leq)$ be a weakly atomic completely semidistributive lattice. 
If $L$ is extremal, there is a bijective correspondence between linear extensions of $\preceq$ in the poset $(Q_L, \preceq)$ and the extremal chains in $L$. 
Given a linear extension $\mathcal{L}$ of $(Q_L, \preceq)$, the set $D(\mathcal{L})$ of downsets of $\mathcal{L}$ yields the extremal chain $\{\bigvee D\mid D \in D(\mathcal{L})\}$.
\end{theorem}

\begin{proof}
Since $L$ is extremal, by corollary \ref{Cor: Q_L being acyclic}, $Q_L$ contains no oriented cycle and the poset $(Q_L, \preceq)$ is well-defined.

To establish the bijection, consider a linear extension $\mathcal{L}$ of the order $\preceq$ induced from the acyclic quiver $Q_L$. The collection $D(\mathcal{L})$ of all downsets of $\mathcal{L}$ forms a totally ordered set by inclusion. It follows from Theorem \ref{Thm: bijection between LM(L) and successor-closed sets} that for each $D \in D(\mathcal{L})$, the element $\bigvee D$ is left modular. Given $j \in \JI^c(L)$, we consider the downset $D(j)$ of $\mathcal{L}$ generated by $j$ and the one, denoted $D(j^-)$, consisting of all elements from $\JI^c(L)$ which are strictly less than $j$ in $\mathcal{L}$. We have $\bigvee D(j^-) 
    \le \bigvee D(j)$ and it follows again from Theorem \ref{Thm: bijection between LM(L) and successor-closed sets} that this is a proper inequality. There is a unique element from $\JI^c(\mathcal{L})$, namely $j$, which belongs to $\bigvee D(j)$ but not to $\bigvee D(j^-)$. Hence, we see that $\bigvee D(j^-) \lessdot \bigvee D(j)$ with $j$ as a label. Since every element of $\JI^c(\mathcal{L})$ appears as a label in the chain  $\{\bigvee D \mid D \in D(\mathcal{L})\}$, the latter is extremal.

    To prove the converse, let $\mathcal{X}$ be an extremal chain in $L$. Hence, every $j \in \JI^c(L)$ appears as a label of a cover relation in $\mathcal{X}$. This induces a total order $\ll$ on $\JI^c(L)$. If $i \ll j$ for this order, then there are $u \lessdot v \le x \lessdot y$ in $\mathcal{X}$, where $i$ labels the cover relation $u \lessdot v$ and $j$ labels the cover relation $x \lessdot y$. Since $x$ is left modular and $j \not\le x$, by Prop. \ref{Prop: Characterization of left modular} we have $x \le \kappa(j)$. Hence, $i \le v \le x \le \kappa(j)$. Thus, there is no arrow from $i$ to $j$ in the labelling quiver. Consequently, the only possible arrows in $Q_L$ are arrows $i' \to j'$, where $i' \gg j'$. This proves $\ll$ is a linear extension of $\prec$.

    From the constructions, it is clear that the correspondences given above are inverses of each other. 
\end{proof}

We finish this section with the following remark and question.

\begin{remark}\label{Remark-Question}
A finite lattice is known to be semidistributive if and only if it is a $\kappa$-lattice, in which case it is automatically well-separated. 
In contrast, if $L$ is an infinite complete lattice, complete semidistributivity of $L$ does not necessarily imply that $L$ is a $\kappa$-lattice, and being a well-separated $\kappa$-lattice does not necessarily imply complete semidistributivity of $L$ (see \cite[Section 3.3]{RST}). These observations lead to the following question, which could be of interest for further investigations.

\textbf{Question:} Find a minimal set of conditions which, together with complete semidistributivity, guarantee that a complete lattice is a well-separated $\kappa$-lattice. 

Such conditions would presumably need to include bi-spatiality, because that is not guaranteed only by complete semidistributivity (see \cite[Example 3.24]{RST}). 
\end{remark}

\bigskip

\section{Brick-directed algebras, left modularity, and extremality}\label{Sec: Brick-directed algebras, Left modularity, and Extremality}
In this section, we return to the initial motivation for our generalization of the left modularity and extremality to (some) infinite lattices, and consider the fruitful setting of lattice of torsion classes. In this contexts, we revisit the extremality and left modularity, which results in alternative realizations of both notions in a purely representation-theoretical language (as in Theorem \ref{thm: characterization by brick-splitting} and Corollary \ref{Cor: left modularity and brick-directed algs}), as well as some nontrivial consequences (including Corollaries \ref{Cor: On tors kQ and Weyl group in Introduction} and \ref{Cor: left modular and extremal lattices of big cardinality in Introduction}). 

\medskip

As noted before, for any finite dimensional $k$-algebra $A$, the set of all torsion classes in $\modu A$, denoted by $(\tors A,\subseteq)$, forms a completely semidistributie weakly atomic well-separated $\kappa$-lattice under the inclusion-order. Moreover, recall that $\brick A$ denotes the set of all (isomorphism  classes) of bricks in $\modu A$ (see Section \ref{Subsec: Torsion Classes}). 
Without loss of generality, we assume $A$ is a connected $k$-algebra, and $\tors A$ denotes the aforementioned lattice of torsion classes. 

\medskip

Recall that $\mathcal{T} \in \tors A$ is said to be \emph{brick-splitting} torsion class if the pair $(\mathcal{T}, \mathcal{T}^{\perp})$ in $\modu A$ splits $\brick A$ into two disjoint subsets; that is, if every $B \in\brick A$ belongs either to $\mathcal{T}$ or $\mathcal{T}^{\perp}$ (and, obviously, not both). We begin by recalling the following useful characterization of the brick-splitting torsion classes which is used below.

\begin{theorem}[{\cite[Theorem 1.1]{AI+}}]\label{Thm: brick-splitting torsion class characterization}
For an algebra $A$, and any torsion class $\mathcal{T}$ in $\modu A$, the following are equivalent:
\begin{enumerate}
    \item $(\mathcal{T}, \mathcal{T}^{\perp})$ is a brick-splitting torsion pair.
    \item Every $B$ in $\brick A$ appears as an arrow in $\Hasse [0,\mathcal{T}]$ or $\Hasse[\mathcal{T}, \modu A]$.
    \item $\mathcal{T}$ is a left modular element of the lattice $\tors A$.
\end{enumerate}
\end{theorem}

The following result is a combination of \cite[Propositions 3.4, 3.5]{AI+}. Alternatively, it can be obtained using Theorems \ref{Thm: bijection between LM(L) and successor-closed sets} and \ref{Thm: brick-splitting torsion class characterization}.
For the definition and properties of the brick-quiver of an algebra $A$, denoted by $Q^b(A)$, see Section \ref{Subsection:Brick-directed Algebras}.

\begin{proposition}
\label{Prop: brick-splitting, left modular, successor-closed}
Let $A$ be a finite dimensional $k$-algebra. 
Then, for any torsion class $\mathcal{T} \in \tors A$, the following are equivalent:
\begin{enumerate}
    \item $\mathcal{T}$ is a left modular element of the lattice $\tors A$.
    \item {The bricks in $\mathcal{T}$ forms a successor-closed subquiver of the brick quiver of $A$.}
\end{enumerate}    
\end{proposition}

From the preceding proposition, together with our results on extremality and left modularity in Sections \ref{Section: Left modularity and extremality for well-separated kappa-lattices} and \ref{Section: Left modularity and Extremality for weakly atomic lattices}, we obtain the following result.

\begin{proposition}\label{Prop: tors A being left modular, extremal, and a chain}
For an algebras $A$, the following are equivalent:
\begin{enumerate}
    \item $\tors A$ is left modular.
    \item $\tors A$ is extremal. 
    \item There is a maximal chain $\{\mathcal{T}_i\}_{i\in I}$ in $\tors A$ where each $\mathcal{T}_i$ is a brick-splitting.
\end{enumerate}
\end{proposition}

\begin{proof}
By Theorems \ref{Thm:RST implications} and \ref{Thm: properties of tors(A)}, $\tors A$ is a complete lattice that is weakly atomic completely semidistributive and well-sparated $\kappa$-lattice. That being the case, the equivalence (1)$\Leftrightarrow$(2) follows from Theorem \ref{Thm: for well-sep kappa lattice, left modular equi to extremal}, or alternatively from Theorem \ref{Thm: Equiv left modular & Extremal for weakly atomic completely semidistributive lattices}. 

Finally, (1)$\Leftrightarrow$(3)
follows from the equivalence between the notions of brick-splitting and left modular torsion classes, from Theorem \ref{Thm: brick-splitting torsion class characterization}.
\end{proof}

Now we give a characterization of extremality and left modularity of $\tors A$ in terms of brick-directed algebras. In particular, we prove Corollary \ref{Cor: left modularity and brick-directed algs}.
Recall that $X_0\xrightarrow{f_1} X_1\xrightarrow{f_2}  \cdots \xrightarrow{f_{m-1}}  X_{m-1}\xrightarrow{f_m}  X_m=X_0$ in $\modu A$ is called a \emph{brick-cycle} if $X_i \in \brick A$, for all $0\leq i\leq m$, and each $f_i$ is a nonzero and non-invertible morphism. Then, $A$ is called \emph{brick-directed} if there exists no brick-cycle in $\modu A$. 

\begin{corollary}
\label{Cor: brick-directed algs and left modularity=extremality}
For an algebra $A$, the following are equivalent:
\begin{enumerate}
    \item $A$ is brick-directed;
    \item $\tors A$ is a left modular lattice;
    \item $\tors A$ is an extremal lattice;
\item $Q^b(A)$ is acyclic.
\end{enumerate}
\end{corollary}

\begin{proof}
By Theorem \cite[Theorem 4.1]{AI+}, algebra $A$ is brick-directed if and only if $Q^b(A)$ is acyclic, if and only if there exists a maximal chain $\{\mathcal{T}_i\}_{i\in I}$ in $\tors A$ consisting of brick-splitting torsion classes. The aforementioned theorem, together with Proposition \ref{Prop: tors A being left modular, extremal, and a chain}, implies the set of desired equivalences.
\end{proof}

From the above characterizations, for finite acyclic quivers, and the Weyl groups of simply laced Dynkin diagrams, we conclude the following result, which proves Corollary \ref{Cor: On tors kQ and Weyl group in Introduction}. We remark that the first part of the next corollary already appeared in our earlier work \cite[Proposition 4.9]{AI+}. However, in that paper, in the absence of a precise lattice-theoretical language for the left modularity and extremality of infinite lattice, this result was formulated in terms of brick-directed algebras.

\begin{corollary}\label{Cor: On tors kQ and Weyl group}
Let $Q$ be a connected acyclic quiver. Then
\begin{enumerate}
    \item The lattice of torsion classes $\tors kQ$ is left modular if and only if $Q$ is a Dynkin quiver or $Q$ is the Kronecker quiver.
    \item For a Dynkin quiver $Q$ and the associated Weyl group lattice  $(W_Q,\leq)$ with respect to the weak order, $\hat{0}$ and $\hat{1}$ are the only left modular elements.
\end{enumerate}
\end{corollary}

\begin{proof}
By \cite[Prop. 4.9]{AI+}, $kQ$ is brick-directed if and only if $Q$ is Dynkin or the Kronecker quiver. This, together with Corollary \ref{Cor: brick-directed algs and left modularity=extremality}, implies (1).

To prove (2), we first note that, from \cite{Mi}, it is known that $(W_Q,\leq)$ with respect to the weak order is isomorphis to $(\tors \Pi(Q), \subseteq)$, where $\Pi(Q)$ denotes the preprojective algebra of $Q$. Furthermore, $\Pi(Q)$ is known to be a self-injective algebra. 
On the other hand, by \cite[Corollary 1.3]{AI+}, if $A$ is a connected self-injective algebra, then the only brick-splitting torsion pairs in $\modu A$ are $(0,\modu A)$ and $(\modu A, 0)$. From these, together with Theorem \ref{Thm: brick-splitting torsion class characterization}, we immediately conclude that $\tors \Pi(Q)$ admits only trivial left modular elements. 
Consequently, the lattice isomorphism between $(\tors \Pi(Q), \subseteq)$ and $(W_Q,\leq)$ implies the desired result.
\end{proof}

Before we state our next result, let us note that in Corollary \ref{Cor: On tors kQ and Weyl group} and its proof, if $Q$ is a non-Dynkin quiver, the corresponding preprojective algebra is not finite dimensional, and hence our results do not apply directly to $\Pi(Q)$ and the Weyl group $(W_Q,\leq)$.

\medskip

Our realizations of the left modularity and extremality in the representation-theoretical context, together with our construction of brick-directed algebras in \cite[Section 4.3]{AI+}, allow us to give an abundance of explicit examples of left modular and extremal lattices of arbitrary length, for any cardinality $\aleph$. 

\medskip

Before we state the next result, let us recall that for every infinite cardinal $\aleph$, there exists an algebraically closed field $k$ of cardinality exactly $\aleph$. 
More specifically, starting from every infinite cardinal $\aleph$ and every characteristic ($0$ or prime $p$), by the ``Existence and Uniqueness of Algebraic Closures", there exists an algebraically closed field of cardinality exactly $\aleph$. 
In fact, if $\aleph=\aleph_0$, the algebraic closures $\overline{\mathbb{Q}}$ and $\overline{\mathbb{F}_p}$ are of cardinality $\aleph_0$, and, respectively of characteristic $0$ and prime $p$. 
And if $\aleph>\aleph_0$, starting from the prime subfield of the characteristic $0$ or $p$ (respectively, $k_0=\mathbb{Q}$ and $k_0=\mathbb{F}_p$), one considers the purely transcendental extension $k:=k_0(x_{\alpha} : \alpha < \aleph)$, by adjoining $\aleph$ independent indeterminates $x_{\alpha}$'s. Conseuqnelty, the filed $k$ and its algebraic closure are of cardinality $\aleph$ (for more details, see \cite[Chapter 13.4]{DF})

\medskip

Now we can verify Corollary \ref{Cor: left modular and extremal lattices of big cardinality in Introduction}.

\newpage

\begin{corollary} \label{Cor: left modular and extremal lattices of big cardinality}
Let $m$ denote a positive integer, and $\aleph$ an infinite cardinality.
\begin{enumerate}
    \item For every $m>4$, there exist infinitely many connected algebras $A$ for which $\tors A$ is a left modular (equivalently, extremal) lattice of length $m$.
    \item For every $\aleph$, there exists a connected algebra $A$ of rank $2$ for which $\tors A$ is left modular and extremal lattice with $2^{\aleph}$ many maximal chain of size ${\aleph}$ consisting of left modular elements.

    \item For every $\aleph$, there exists an extremal lattice for which $\LM(L)$ is of size $2^{\aleph}$. 
\end{enumerate}
\end{corollary}

\begin{proof}
We first note that the gluing construction in \cite[Section 4.3]{AI+} allows us to start from two brick-directed algebras $\Lambda_{m_1}$ and $\Lambda_{m_2}$, where $\tors \Lambda_{m_1}$ and $\tors \Lambda_{m_2}$ are respectively of length $m_1$ and $m_2$, and then obtain a new brick-directed algebra $A$ where $\tors A$ is of length $m_1+m_2-1$. 
In particular, since $A$ is brick-directed, Corollary \ref{Cor: brick-directed algs and left modularity=extremality} implies that $\tors A$ is left modular and extremal, and furthermore it is of length $|\brick A|=|\brick \Lambda_{m_1}|+|\brick \Lambda_{m_2}|-1$. 

Second, we consider the local algebra $L_n= k[x_1,\cdots, x_n]/\langle x_ix_j : 1\leq i\leq j \leq n \rangle$, whose bound quiver consists of $n$ loops that start and end at the same vertex, and it is subject to all quadratic relations. Note that $L_1$ is representation-finite, $L_2$ is (representation-infinite) tame, and $L_n$ is wild for all $n>2$. Moreover, $\brick L_n$ consists of the simple module associated to the unique vertex. Hence, $L_n$ is trivially brick-finite and brick-directed. In particular, gluing $L_n$ to any nontrivial brick-directed algebra $\Lambda$ results in a new (non-isomorphic) brick-directed algebra $\Lambda({L_n})$, where $\brick \Lambda(L_n)=\brick \Lambda$.  

\medskip

For (1), to employ the above construction, we give a list of explicit brick-directed algebras $\Lambda_{m_i}$ for which $\tors \Lambda_{m_i}$ is of length ${m_i}$ (equivalently, $|\brick \Lambda_{m_i}|=m_i$). Then, we obtain the desired algebra $A$ from the gluing of two algebras $\Lambda_{m_i}$ and $\Lambda_{m_j}$, for some ${m_i}$ and ${m_j}$. 
As observe below, we only need to specify three algebras $\Lambda_3$, $\Lambda_{5}$ and $\Lambda_6$, that is, three connected brick-directed algebras algebras, respectively with $3$, $5$, and $6$ bricks.  
Then, for every other $m>4$, we use the gluing construction to obtain a brick-directed algebra $A$ with $|\brick A|=m$, thus by Corollary \ref{Cor: brick-directed algs and left modularity=extremality}, $\tors kA$ is a left modular and extremal lattice of length $m$.

By Corollary \ref{Cor: On tors kQ and Weyl group}, for any quiver $A_n$, the path algebra $kA_n$ is brick-directed; thus $\tors kA_n$ is left modular and extremal. 
Also, length of $\tors kA_n$ is the same as the size of $\brick kA_n$, which is $n(n+1)/2$. 
In particular, for $m=3$, consider $\Lambda_3:=kA_2$, and the lattice of $\tors \Lambda_3$, which is of length $3$ (see Figure \ref{Fig. for A_2}). Similarly, for $m=6$, we consider $\Lambda_6:=kA_3$, and $\tors \Lambda_6$. 
In contrast, for $m=5$ we no longer have a (connected) path algebra; instead we consider the algebra $\Lambda_5=kQ/I$, where $Q: 1\xrightarrow{\alpha} 2\xrightarrow{\beta} 3$, and $I=\langle \beta\alpha \rangle$. 
Because $\Lambda_5$ is a quotient of $kA_3$, it is again brick-directed (see \cite[Corollary 1.8]{AI+}). In particular, $\ind A= \brick A =\{{\begin{smallmatrix} 1\end{smallmatrix}}, {\begin{smallmatrix} 2\end{smallmatrix}}, {\begin{smallmatrix} 3\end{smallmatrix}}, {\begin{smallmatrix} 1\\2\end{smallmatrix}}, {\begin{smallmatrix} 2\\3\end{smallmatrix}} \}$, and $\tors \Lambda_5$ is left modular of length $5$.

With the three explicit brick-directed algebras $\Lambda_3$, $\Lambda_5$, and $\Lambda_6$ in hand, we can apply the gluing construction to generate a brick-directed algebra $\Lambda_m$, for any $m>4$. For instance, for $m=7$, the brick-directed algebra $\Lambda_7$ is obtain from gluing $\Lambda_3$ and $\Lambda_5$, which admits $7=3+5-1$ bricks, hence $\tors \Lambda_7$ is left modular and extremal of length $7$. 
Furthermore, we can glue the brick-directed algebra $\Lambda_m$ to the local algebra $L_n$, to obtain $\Lambda_m(L_n)$. Note that, for every $n \in \mathbb{Z}_{>0}$, we have a new brick-directed algebra $\Lambda_m(L_n)$ with exactly $m$ bricks. 

\medskip

For (2), consider the Kronecker quiver $Q$ and the associated path algebra $A=kQ$. Obviously, $A$ is of rank $2$ and, by Corollary \ref{Cor: On tors kQ and Weyl group}, $A$ is brick-directed, hence $\tors A$ is left modular and extremal (see Corolarry \ref{Cor: brick-directed algs and left modularity=extremality}). It is known that $\brick A$ contains an infinite family of pairwise Hom-orthogonal modules $\{B_{\lambda}\}_{\lambda \in \mathbb{P}^1(k)}$, where $\dim_k(B_{\lambda
})=2$. 
Moreover, as remarked before, for every infinite cardinal $\aleph$, there exists an algebraically closed field $k$ of cardinality exactly $\aleph$. Hence, there exist $2^{\kappa}=\kappa^{\kappa}=\kappa !$ partial order that we can consider on $\{B_{\lambda}\}_{\lambda \in \mathbb{P}^1(k)}$. Using these partial order and constructing a chain of brick-splitting torsion classes, we then add the preinjective and preprojective bricks, and (by Zorn's Lemma) extend that to total order on $\brick A$, to obtain $2^{\kappa}$ chains of brick-splitting torsion classes. Each of these maximal chains in $\tors A$ consists of brick-splitting torsion classes.

\medskip

Part (3) follows from the proof of (2). More specifically, for the Kronecker algebra $A=kQ$ considered in part (2), by Theorem \ref{Thm: properties of tors(A)} the lattice $L=\tors A$ satisfies all the assumptions of Prop. \ref{Prop: LM(L) is distributive sublattice}. 
Thanks to the Auslander-Reiten quiver $\Gamma_{kQ}$ of $\modu kQ$, we can explicitly describe the set $\brick kQ = \{B_{p}\}_{p \in \mathcal{P}} \cupdot\{B_{\lambda}\}_{\lambda \in \mathbb{P}^1(k)} \cupdot \{B_{i}\}_{i \in \mathcal{I}}$, where $B_p$, $B_\lambda$, and $B_i$ respectively denote the bricks belonging to the preprojective component $\mathcal{P}$, quasi-simples on the mouth of tubes, and those in the preinjective component $\mathcal{I}$. We further observe that $\mathcal{P}$ and $\mathcal{I}$ are countable infinite sets, thus of size $\aleph_0$. Consequently, $\brick kQ$ is of size $\aleph$. 
Also, note that elements of $\brick kQ$ are in bijection with the left modular elements in $\tors kQ$, which are the brick-splitting torsion classes in $\tors kQ$; see Theorem \ref{Thm: brick-splitting torsion class characterization}. Hence, the assertion follows.
\end{proof}

In the following remark, we collect some observations to supplement the construction in the previous proof.

\begin{remark}
With regard to Corollary \ref{Cor: left modular and extremal lattices of big cardinality}, we note that
\begin{enumerate}
    \item The algebra constructed in the proof of part (1) is in fact a very simple gentle algebra, obtained as a sequence of gluing of quivers of type $A_n$. We also note that $\Lambda_5$ could be obtained from the gluing of two copies of $\Lambda_3$, but we explicitly constructed that for the sake of clarity. 
    It is notable that, for any such gentle algebra $A$ constructed in the proof, computation of $\tors A$ can be conducted in a very combinatorial way, as described in \cite{BD+}, and there exist Applets which compute the lattice $\tors A$ with the labelling of the cover relations by bricks (see \cite{En, Ge}).
    
    \item For the path algebra $kA_1$, we have $\tors kA_1$ is of length $1$. However, there is no (connected) brick-directed algebra $A$ for which $\tors A$ is of length $2$ or $4$. 
    Using the notation of the previous proof, we assume that $\Lambda_m$ denotes a brick-directed algebra with $|\brick \Lambda_m|=m$. 
    Now, observe that, for connected algebras, the ``no gap" phenomenon in the existence of $\Lambda_m$ begins from $m=5$. 
    Obviously, if we include disconnected algebras, then we have an algebra for any positive integer $m$.
    \end{enumerate}
\end{remark}

To put our lattice-theoretical results into a broader perspective and highlight their connections with some fundamental phenomena treated in other active areas of research, we remark on some aspects of Corollary \ref{Cor: left modular and extremal lattices of big cardinality} and its proof, in connection with some open and challenging conjectures in modern representation theory.
 
\begin{remark}\label{Remark-Question: on the infinite cardinal}
We note that Corollary \ref{Cor: left modular and extremal lattices of big cardinality} closely relates to the study of infinite \emph{semibricks}; that is, an infinite set of pairwise Hom-orthogonal bricks in $\brick A$. To state that more accurately, we make the following observations:

\begin{enumerate}
    \item In the proof of part (2), we use the well-known structure $\tors kQ$, for the Kronecker quiver $Q$. However, we can alternatively use the generalized barbell algebra of rank $2$, which is also a gentle algebra; see \cite[Section 1]{Mo}. 
    In fact, a decisive point in the proof is the existence of an infinite semibrick, which is known to always exist for any brick-infinite algebra which is hereditary or biserial; see \cite{MP1}. Another key component of the proof is the brick-directedness, which guarantees (and is equivalent to) left modularity of $\tors A$. As noted in Section \ref{Subsection:Brick-directed Algebras}, brick-directed algebras conceptually generalize classical representation-directed algebras. It is known that for any rep-directed algebra $A$, the lattice $\tors A$ is left modular and extremal, but it is always  a finite lattice; \cite[Corollary 1.5]{TW}. 
    However, unlike rep-directed algebras, which are always of rep-finite type, brick-directed algebras include various types of algebras, of tame and wild types; \cite[Corollary 1.6]{AI+}.

    \item In connection with some important open questions in representation theory, we formulate the following problem in a more lattice-theoretical setting.

    \textbf{Question}: Let $k$ be an algebraically closed field of cardinality $\aleph$, and assume that $A$ is a brick-directed $k$-algebra. Are the following equivalent?
    \begin{enumerate}
        \item [(I)] $A$ is a brick-infinite algebra. 
        \item [(II)] There exist $2^{\aleph}$ maximal chains in $\tors A$, each of size ${\aleph}$ and consisting of left modular elements.
    \end{enumerate}
    
\medskip
Note that (II) $\Rightarrow$ (I) follows from the definition, but (I) $\Rightarrow$ (II) is unknown to us. 
It is notable that this open implication closely relates to the Strong Semibrick Conjecture; see \cite[Remark 3.5]{MP2}. 
Although this conjecture is still open in full generality, it is known to hold for all hereditary and all biserial algebras (and some other families). Hence, for these families, we can give an affirmative answer to the above questions. For several closely related open conjectures on bricks, see \cite{MP2} and references therein. 
\end{enumerate}
\end{remark}

\medskip

\textbf{Acknowledgments.} We thank Hugh Thomas, particularly for sharing some ideas on the content of Section \ref{Section: Left modularity and extremality for well-separated kappa-lattices}. KM also thanks Adrien Segovia for some interesting discussions. 
SA was supported by Early-Career Scientist JSPS Kakenhi grant number 23K12957. 
OI was supported by JSPS Grant-in-Aid for Scientific Research (B) 22H01113, (B) 23K22384.
KM was supported by Early-Career Scientist JSPS Kakenhi grant number 24K16908. CP was supported by the Natural Sciences and Engineering Research Council of Canada (RGPIN-2026-05988) and by the Canadian Defence Academy Research Programme. 
Part of this project developed while SA, OI, and KM were attending the thematic program TDA PARTI: {\small{Topological Data Analysis, Persistence And Representation Theory Intertwined}}, hosted by the OIST's Visiting Program, at Okinawa Institute of Science and Technology.

\end{document}